\documentclass[11pt]{article}
\usepackage{amsmath,amsfonts,amssymb,amsthm,enumerate}
\usepackage{graphicx}
\usepackage{caption}
\usepackage{tikz,authblk}
\usepackage{subfigure}
\usepackage{hyperref}

\theoremstyle{definition}
\newtheorem{definition}{Definition}[section]
\newtheorem{example}[definition]{Example}

\theoremstyle{remark}

\theoremstyle{plain}
\newtheorem{theorem}[definition]{Theorem}
\newtheorem{lemma}[definition]{Lemma}
\newtheorem{proposition}[definition]{Proposition}

\newtheorem{corollary}[definition]{Corollary}
\newtheorem{conjecture}[definition]{Conjecture}
\newtheorem{qn}[definition]{Question}

\begin{document}

\author[1] {Basudeb Datta}
\author[2] {Dheeraj Kulkarni}

 \affil[1] {Department of Mathematics, Indian Institute of Science, Bangalore 560\,012, India.  dattab@math.iisc.ernet.in.}
 \affil[2] {Theoretical Statistics and Mathematics Unit, Indian Statistical Institute, 203 B. T. Road, Kolkata 700\,108,  India.  dheeraj.kulkarni@gmail.com}

\title{Minimal contact triangulations of 3-manifolds}

\date{August 12, 2016}

\maketitle

\vspace{-10mm}

\begin{abstract}
In this paper, we explore minimal contact triangulations on contact 3-manifolds. We give many explicit examples of contact triangulations that are close to minimal ones. The main results of this article say that on any closed oriented 3-manifold the number of vertices for minimal contact triangulations for overtwisted contact structures grows at most linearly with respect to the relative $d^3$ invariant. We conjecture that this bound is optimal. We also discuss contact triangulations for a certain family of overtwisted contact structures on 3-torus.
\end{abstract}

\section{Introduction and Main Results}

The triangulations of manifolds have played great role in understanding the topology of the underlying manifold.
Especially, in 3-manifold topology the study of triangulations has been a driving force behind many classical results
which are qualitative in nature. Given a manifold, the study minimal triangulations (triangulations with least number
of vertices or edges and so on) is a part of combinatorial topology. Minimal triangulations are closely tied up with
the topology of the underlying manifold.

A contact structure on a 3-manifold $M$ is a 2-plane field that is nowhere integrable.
Very often, a contact structure is given by $\text{ker}(\alpha)$ for some 1-form $\alpha$ on the manifold $M$. By Frobenius' theorem, the non-integrability of $\text{ker}(\alpha)$ is equivalent to $\alpha \wedge d \alpha \neq 0$. In this case, we call $\alpha$ as a contact form. Throughout the article we will consider contact structures that are defined by kernel of some 1-form. For instance, on $\mathbb{R}^3$ the 1-form $\alpha_{st} = dz +xdy$ defines a contact structure. We denote $\text{ker}(\alpha_{st})$ by $\xi_{st}$.

In the recent past, the contact structures have been studied with great interest for various reasons. In dimension three, there is a tremendous progress (see for example \cite{DG}, \cite{El89},  \cite{EH}, \cite{Gi00}) in understanding the contact structures. To begin with, it is known that every closed oriented 3-manifold admits
a contact structure (see \cite{Lu}, \cite{Ma}).

The notion of cellular decomposition that is compatible with contact structure has been used in \cite{Gi00} and \cite{HKM}. Thi is formally known as {\em contact cellular decomposition} (see \cite[Section 1.1]{HKM}, \cite[Section 4.7]{Et}). It has been used to show existence of open books and partial open books. Inspired from this notion, we take the following definition of contact triangulation.

Let $(M, \xi)$ be a contact 3-manifold. A triangulation $X$ of $M$ is said to be a {\em contact triangulation} if
(i) all the edges of $X$ are Legendrian (i.e., edges are tangent to contact plane at all points), (ii) no 2-face of
$X$ is an overtwisted disk (see Section \ref{ot_disk_dfn} for the definition of overtwisted disk), and (iii) the
restriction of $\xi$ to the interior of each 3-face of $X$ is contactomorphic to the standard contact structure
$(B, \xi_{st})$ for some open set $B \subset \mathbb{R}^3$. Hence, no 3-face contains an overtwisted disk.

We note that this definition differs from that of contact cellular decompostion on the following point.
In the earlier notion, the 2-faces are required to satisfy the condition that the twisting number $tw$ is $-1$ along
the boundary of each 2-face. We relax this condition as we need to subdivide the triangulation in order to achieve $tw = -1$.
From the combinatorial view point, this imposes additional problem as the growth of minimum number of vertices is exponential
and difficult to keep track of. Instead, with the above notion, as we will show later, it is easier to capture geometric
complexity of contact structures. For an explicit example of a contact triangulation
see Example \ref{explicit_example}.

As noted before, the notion of contact triangulations (contact cellular decomposition to be precise) is used to obtain qualitative results.
On the other hand, quantitative aspect of contact triangulations has not been studied before  as far as we know. We say that a contact
triangulation $\Sigma$ of $(M, \xi)$ is \emph{minimal}  if the number of vertices in $\Sigma$ gives the lower bound
for the number of vertices for any  contact triangulation of $(M, \xi)$. One expects that minimal contact triangulations
should reflect the complexity of the contact manifold. Thus, the following questions are of interest.

\begin{qn}
Given a contact 3-manifold $(M, \xi)$, what is the minimum number of vertices $n(\xi) $ required to have a contact triangulation of $(M, \xi)$?
How does the number $n(\xi) $ change if we change $\xi $ on $M$?
\end{qn}

In general it is hard to give examples of minimal triangulations. The same holds true of minimal contact triangulations.

There is dichotomy of contact structures. The notion of overtwisted and tight (see Section \ref{ot_disk_dfn}) contact structures
was introduced by Eliashberg in \cite{El89}. The classification (up to isotopy) of contact structures is relatively easier
for overtwisted contact structures than tight ones. Eliashberg showed that overtwisted contact structures are classified
completely by the homotopy classes of contact structures as 2-plane fields. However, existence and classification
problem of tight contact structure is not completely understood in generality.

If we take a triangulation of $(M, \xi)$ with $\xi$ tight then we can deform
it into a contact triangulation by applying small perturbations (see Lemma \ref{tight_triangulation}). Thus,
for tight contact structures minimal contact triangulations coincide with minimal triangulations of the underlying manifold.
However, for overtwisted contact manifolds, minimal contact triangulations turn out to be interesting objects.
Before we state the main results, we fix some notation. Given a triangulation $\Sigma$ of a 3-manifold $M $, $f_0(\Sigma)$
denotes the number of vertices in $\Sigma$. The following result produces contact triangulations efficiently for all
overtwisted contact structures on 3-sphere by using Lutz twist operation (see Subsection \ref{lt}).

\begin{theorem}\label{mainthm-1}
Let $\xi$ be a contact structure on $\mathbb{S}^3$. Let $\xi_+$ denote the overtwisted contact structure obtained by
Lutz twisting along a positively transverse right handed trefoil in $(\mathbb{S}^3, \xi)$ with self-linking number $+1$.
Let $\xi_-$ denote the overtwisted contact structure obtained by Lutz twisting along a positively
transverse unknot in $(\mathbb{S}^3, \xi)$ with self-linking number $-1$. If $\Sigma$ is a contact triangulation of $(\mathbb{S}^3, \xi)$,
then there exists a contact triangulation $\Sigma_+$ $($respectively $\Sigma_-)$ of $(\mathbb{S}^3, \xi_+)$ $($respectively
$(\mathbb{S}^3, \xi_-))$ such that $f_0 (\Sigma_{\pm})= f_0(\Sigma) + 3$.
\end{theorem}

We consider $\mathbb{S}^3$ as the boundary of the unit ball in $\mathbb{R}^4 $. Let $(x_1,y_1,x_2,y_2)$ denote a coordinate
system in $\mathbb{R}^4$. Then, the 2-plane distribution given by $\text{ker}(x_1dy_1 -y_1dx_1 + x_2dy_2 -y_2dx_2) $
defines a contact structure on $\mathbb{S}^3$. We denote it by $\xi_{std}$.
As a consequence of the above theorem, we get 

\begin{corollary}\label{cor-1}
Let $\xi_{ot}$ be an overtwisted contact structure on $\mathbb{S}^3$ with
$d^3(\xi_{ot}, \xi_{std}) = n$ for some $n \in \mathbb{Z}$. Then there exists a
contact triangulation of $(\mathbb{S}^3, \xi_{ot})$ with $m$-vertices where
$$
m=
\begin{cases}
3|n| +4 \ \text{ if } n \neq 0, \\
10 \ \text{ if } n =0.
\end{cases}
$$
\end{corollary}

We have a generalization of Theorem \ref{mainthm-1} for a contact 3-manifold $(M, \xi)$ as follows.

\begin{theorem}\label{mainthm-2}
Let $M$ be a closed $3$-manifold. Let $\xi$ be a contact structure on $M$. We have a positively transverse
right-handed trefoil $K$ and unknot $U$ embedded in a small Darboux ball in $(M, \xi)$ with self-linking
numbers $+1$ and $-1$ respectively. Let $\xi_+$ $(\xi_-)$ denote the contact structure obtained by Lutz-twisting
$\xi$ along $K$ $($respectively along $U)$. Let $\mathcal{M}$ be a contact triangulation of $(M, \xi)$. Then
there exists a contact triangulation $\mathcal{M}_+$  (respectively $\mathcal{M}_-)$ of $(M, \xi_+)$ $($respectively
$(M, \xi_-))$ such that $f_0(\mathcal{M}_{\pm}) = f_0(\mathcal{M}) +3$.
\end{theorem}

As a consequence we get 

\begin{corollary}\label{cor-2}
Let $(M, \xi)$ be a closed contact $3$-manifold with a contact triangulation $\mathcal{M}$. Let $\xi'$ be an
overtwisted contact structure with $d^2(\xi', \xi)=0$ and $d^3(\xi', \xi)=n$ for some $n \in \mathbb{Z} $.
Then there exists a contact triangulation of $\mathcal{M}'$ of $(M, \xi')$ such that $f_0(\mathcal{M}')= 3 |n|+ f_0(\mathcal{M})$.
\end{corollary}

The above results give an upper bound (linear in terms relative $d^3$ invariant) for $n(\xi) $ (the number of vertices in minimal contact
triangulations) for overtwisted contact structures $\xi $ with relative $d^2 $ invariant fixed.
It is hard to estimate the growth of vertices for minimal contact triangulation with respect to relative $d^2 $ invariants.
The change in relative $d^2$ of contact structures amounts to Lutz twisting along transverse knots that are homologically
nontrivial and could be highly complex from combinatorial view point. Therefore, one can not expect a linear
upper bound for $n(\xi) $ in terms of relative $d^2 $ invariant. To explore the effect of Lutz twisting
along a homologically nontrivial knots, we consider a certain family of contact structures $\xi_n $ on $3 $-torus
(see Section \ref{3-torus} for details) that are obtained by repeated application of Lutz twist along a
fixed knot which is not null-homologous. We construct contact triangulations of this family of contact structures.
This triangulations are {\em close} to minimal ones by virtue of the construction.
Specifically, we have the following

\begin{theorem}\label{3-torus_thm}
 For $n \geq 2$, there is a contact triangulation of $(T^3, \xi_n)$ with $8n^3$ vertices and $40n^3$ \, $3$-simplices.
\end{theorem}

\bigskip

\noindent {\bf Acknowledgements:} The first author was partially supported by DIICCSRTE, Australia (project AISRF06660) \& DST, India (DST/INT/AUS/P-56/2013 (G)) under the Australia-India Strategic Research Fund.
A major part of the work was done when the first author visited Stat-Math Unit,
Indian Statistical Institute (ISI), Kolkata during 2016 summer. He thanks ISI for hospitality. The second author thanks Prof. Mahan Mj.
for the support via J.~C. Bose Fellowship when the first part of this work was carried out. He also thanks ISI, Kolkata for
the support during the later part of the work.

\section{Preliminaries} \label{ot_disk_dfn}

In this section, we recall some notions that are relevant to the article. We also state some useful results with references.

Let $(M, \xi)$ be a contact 3-manifold. A smooth path $\gamma$ is called \emph{Legendrian} if $\gamma'(p) \in \xi_{\gamma(p)}$
for all points $p$ in the domain of the definition. Given an embedded closed surface $\Sigma$ (or with Legendrian boundary),
the singular line field $T_x \Sigma \cap \xi_x$ gives a singular foliation on the surface $\Sigma$. We call this singular foliation
as the \emph{characteristic foliation} on $\Sigma$. An overtwisted disk in $(M, \xi)$ is an embedded disk with Legendrian boundary
whose characteristic foliation is as shown in Figure \ref{otdisk}. A contact manifold $(M, \xi)$ is called \emph{overtwisted} if
there is an embedded overtwisted disk otherwise we call it \emph{tight}.

\begin{figure}[!ht]
\centering
\includegraphics[scale=.5]{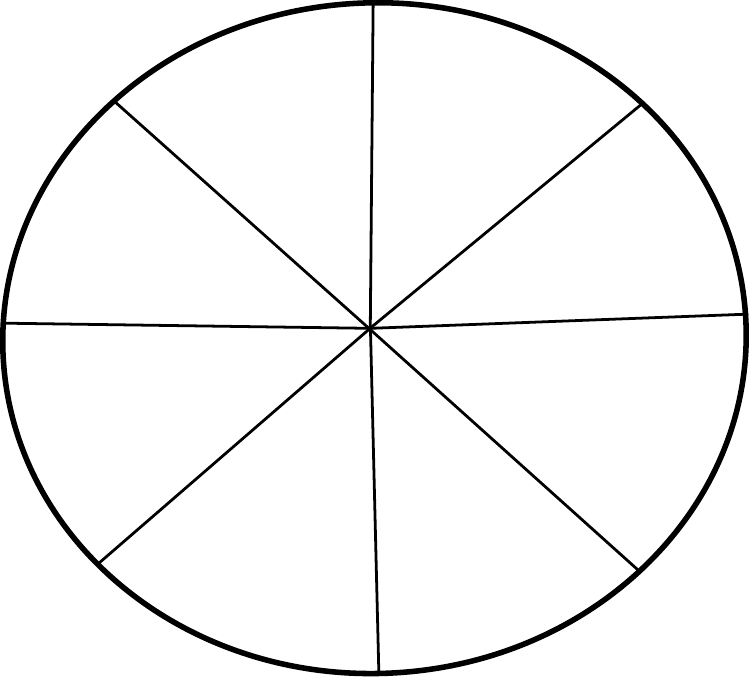}
\caption{Characteristic foliation on an overtwisted disk} \label{otdisk}
\end{figure}

We state following weaker version of Eliashberg's classification result for overtwisted
contact structures on 3-manifolds. This follows from his original result.

\begin{theorem}[Eliashberg, \cite{El89}]
Let $M$ be a closed oriented $3$-manifold. Every homotopy class of oriented and cooriented $2$-plane field on $M$ contains
an overtwisted contact structure. Any two overtwisted contact structures in the same homotopy class of $2$-plane field are
isotopic through contact structures.
\end{theorem}

Thus, the classification of overtwisted contact structures is same as the classification of homotopy classes of
contact structures as 2-plane fields. To classify homotopy classes of oriented and cooriented 2-plane fields,
we need $d^2$ and $d^3$ invariants from obstruction theory. For any two 2-plane fields $\xi$ and $\eta$  the
invariant $d^2(\xi, \eta) \in H^2(M; \mathbb{Z})$ measures the relative obstruction for homotopy between $\xi$
and $\eta$ on 2-skeleton and $d^3(\xi, \eta) \in H^3(M; \mathbb{Z})$ invariant is an obstruction to extension
of homotopy of 2-plane fields onto 3-skeleton. We have the following fact from the obstruction theory.
See \cite[Section 4.2.3]{Ge} for details.

\begin{proposition}
Let $M$ be a closed oriented $3$-manifold. Then two $2$-plane fields $\xi$ and $\eta$ on $M$ are homotopic
if and only if $d^2(\xi, \eta)=0$ and $d^3(\xi, \eta)=0$.
\end{proposition}

\subsection{Lutz Twist}\label{lt}

There is a geometric operation, called as \emph{Lutz twist}, that turns any contact structure into an
overtwisted contact structure without changing the underlying manifold. A knot $K$ in $(M, \xi)$ is
called \emph{transverse} knot if it is transverse to contact structure at any point on the knot. Given a
transverse knot $K$ in $(M, \xi)$, there is a neighborhood $N(K)$ of $K$ which is contactomorphic to
$(\mathbb{S}^1 \times D^2, \text{ker}(d\theta + r^2 d\phi))$, where $\theta$ is the coordinate along
$\mathbb{S}^1 $-direction and $(r, \phi)$ are polar coordinates on the disk $D^2$. We replace the contact
structure on this neighborhood by specifying a contact form $\alpha= h_1(r) d\theta + h_2(r) d \phi$ with
$h_1(r)$ and $h_2(r)$ satisfying the following conditions
\begin{enumerate}[(i)]
 \item $h_1(r) = -1$ and $h_2(r)= 0$ near $r =0$,
 \item $h_1(r)= 1$ and $h_2(r)=r^2 $ near $r=1$,
 \item $(h_1(r), h_2(r)) $ is never parallel to $(h'_1(r), h'_2(r)) $ for $r \neq 0$.
\end{enumerate}
Such a pair of functions exists. (See \cite[Section 4.3]{Ge} for details.) The change in geometry of
contact structure near the transverse knot can be described as follows. The new contact structure is
overtwisted. All the meridional disks of $N(K)$ are now overtwisted disks. Positively transverse knot
$K$ now becomes a negatively transverse knot in the new contact structure. This operation is called as
Lutz twist along $K$. We say that the new contact structure is obtained by Lutz twisting along $K$ from
$\xi$ and it will be denoted by $\xi^K $.


We state the change in the homotopy class of 2-plane field when a Lutz twist is performed along $K$.

\begin{proposition} [{\cite[Section 4.3.3]{Ge}}]
Let $K$ be a positively transverse knot in a closed oriented contact $3$-manifold $(M, \xi)$. Then
$d^2(\xi^K, \xi)= - {\rm PD}([K]) $, where ${\rm PD}([K])$ denotes the Poincar\'{e} dual of the homology class $[K]$.
\end{proposition}

In particular, we note that if $K$ is a null-homologous knot in $M$ then $d^2(\xi^K, \xi)=0 $. We will
use this fact later in the article.

Assume that $K$ is a null-homologous transverse knot in $(M, \xi)$. Then the self-linking number $\text{sl}(K) $ is well-defined (see \cite[Definition 3.5.28]{Ge}). The Lutz twist along such $K$ changes the relative $d^3 $-invariant as follows.

\begin{proposition}[{\cite[Remark 4.2.6]{Ge}}]
Let $K$ be a positively transverse null-homologous knot in $(M, \xi)$. Then $d^3(\xi^K, \xi) = {\rm sl}(K)$.
\end{proposition}

To compute self-linking number of a transverse knot $K$ in $(\mathbb{R}^3, \xi_{st})$, we need the following known fact (for instance see \cite[Proposition 3.5.32]{Ge}).

\begin{proposition}
The self linking number $sl(K)$ equals the writhe of the front projection of $K$.
\end{proposition}

Recall that the {\em writhe} of an oriented knot diagram is given by number of positive crossings minus the number of negative crossings and the {\em front projection} of $K$ in $(\mathbb{R}^3, \xi_{st})$ is the projection of $K$ onto $y-z$ plane. Thus, in $(\mathbb{S}^3, \xi_{std})$, $d^3(\xi^K, \xi)$-invariant equals the writhe of the
front projection of $K$. Figure \ref{fig:trefoil} shows that the writhes of unknot $U$ and right-handed trefoil $K$ are $-1$ and $+1$ respectively. We will use this elementary observation later.

\begin{figure}[!ht]
\centering{
\resizebox{70mm}{!}{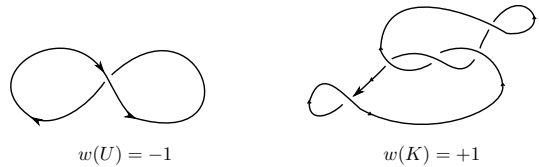}
\caption{Writhe of the front projection of the unknot and right-handed trefoil.} \label{fig:trefoil}
}
\end{figure}

The following statements about contact triangulations will be useful later.

\begin{lemma}
Let $(M, \xi) $ be a closed contact $3$-manifold. Then there exists a contact triangulation of $(M, \xi)$.
\end{lemma}

\begin{proof}
The proof is a slight modification of the argument given in the proof of Lemma 4.8 in \cite{Et}. Consider a finite cover
of $M$ by Darboux charts. We take a triangulation $\Sigma$ of $M$ such that every 3-face lies in some Darboux chart.
Thus, the conditions (ii) and (iii) of the definition are satisfied as $\xi $ restricted to each Darboux ball is tight.
Now we perturb the 2-faces to make the edges Legendrian using weak form of Bennequin inequality (for details, see the proof
of \cite[Lemma 4.8]{Et}). Thus, condition (i) is also satisfied.
\end{proof}


\begin{lemma}\label{tight_triangulation}
Let $(M,\xi)$ be a tight contact manifold. Let $\Sigma$ be any triangulation. Then the $2$-faces of $\Sigma$ can be perturbed
slightly so that we get a contact triangulation $\Sigma'$ of $(M, \xi)$ which is isomorphic to $\Sigma$ $($i.e., there is a
self-diffeomorphism of $M$  taking $\Sigma'$ to $\Sigma)$.
\end{lemma}

\begin{proof}
As $\xi$ is tight, condition (iii) of the definition is satisfied. We make edges Legendrian by perturbing $2$-faces slightly as before.
\end{proof}

\section{Examples}

In this section, we describe examples of contact triangulation with details. Some of the examples will be useful in the later part of the article.
We identify a simplicial complex by the set of its maximal simplices. We also identify a simplex with the set of vertices in it.

We know that the boundary complex $S^3_5=\partial \Delta^4$ of a 4-simplex $\Delta^4$ is a 5-vertex triangulation of $\mathbb{S}^3$.
It is unique and minimal. Therefore, by Lemma \ref{tight_triangulation}, we can get a vertex minimal (with 5 vertices) contact
triangulation of $(\mathbb{S}^3, \xi_{std})$. In Example \ref{explicit_example}, we present an explicit contact triangulation of
$(\mathbb{S}^3, \xi_{std})$ with 8 vertices.

\begin{example}\label{explicit_example}
We write $\mathbb{S}^3 = \{(x_1, y_1, x_2, y_2) : x_1^2 + y_1^2 + x_2^2 + y_2^2 =1\}$ for the unit sphere in $\mathbb{R}^4$.
Consider the $1$-form on $\mathbb{R}^4$
 \begin{align*}
\beta &= x_2dx_1-x_1dx_2 + y_2 dy_1 -y_1dy_2.
\end{align*}
Notice that $d \beta $ is a symplectic form on $\mathbb{R}^4 $ and the radial vector field $Z= x_1 \frac{\partial}{\partial x_1}
+ y_1 \frac{\partial}{\partial y_1} + x_2 \frac{\partial}{\partial x_2} + y_2 \frac{\partial}{\partial y_2} $ is a Liouville vector
field (i.e., $\mathcal{L}_Z d \beta = d \beta $). So, the 1-form $i_z d \beta = \beta$ restricts to a contact form on $\mathbb{S}^3$.
Let $\xi$ denote the contact structure given by $\text{ker} \beta \cap T \mathbb{S}^3  $. Then
\begin{align*}
\xi & =\text{Span} \left(\left \{ y_1 \frac{\partial}{\partial x_1} - x_1 \frac{\partial}{\partial y_1} - y_2 \frac{\partial}{\partial x_2} + x_2
\frac{\partial}{\partial y_2} , -y_2 \frac{\partial}{\partial x_1} + x_2 \frac{\partial}{\partial y_1} - y_1 \frac{\partial}{\partial x_2} + x_1 \frac{\partial}{\partial y_2} \right \}\right).
\end{align*}
The unit ball $B^4$ with the symplectic form $d \beta$ is a strong filling of the contact sphere $(\mathbb{S}^3, \xi)$.
Therefore, $\xi$ is a tight contact structure. By the uniqueness of the tight contact structure on 3-sphere due to Eliashberg,
we see that $\xi$ is isotopic to $\xi_{std}$.

Let $e_i$, for $1 \leq i \leq 4$, denote the unit vector in $\mathbb{R}^4$ whose $i$-th component is $1$ in the standard coordinate
system. Now, consider the triangulation $\Sigma$ of $\mathbb{S}^3$ whose vertex set is
$ V= \{u_1 =e_1, u_2 =-e_1, v_1 =e_2, v_2= -e_2, w_1=e_3, w_2=-e_3, z_1=e_4, z_2= -e_4 \}$ and the set of 3-faces
(these are spherical in shape) is $\{ u_i v_j w_k z_{\ell} : 1 \leq i, j, k ,l \leq 2 \}$, where $u_i v_j w_k z_{\ell}
= \{(x_1, y_1, x_2, y_2) \in \mathbb{S}^3 : (-1)^{i-1}x_1,(-1)^{j-1} y_1, (-1)^{k-1}x_2, (-1)^{l-1}y_2 \geq 0\}$.

The group ${\rm Aut}(\Sigma)$ of automorphisms of $\Sigma$ is isomorphic to $S_4 \times (\mathbb{Z}_2)^4$ (where $S_4$
denotes the group of permutations on $4$ symbols) and is generated by
\begin{align*}
 \{(u_1,v_1)(u_2,v_2), (u_1,w_1)(u_2,w_2), (u_1,z_1)(u_2,z_2), (u_1,u_2), (v_1,v_2), (w_1,w_2), (z_1,z_2)\}.
\end{align*}
This group ${\rm Aut}(\Sigma)$ acts transitively on the set of edges (respectively, 2-faces). Also, each automorphism of
$\Sigma$ induces a diffeomorphism of $\mathbb{S}^3$ in an obvious way and preserves the contact structure $\xi$. Hence,
${\rm Aut}(\Sigma)$ acts on $(\mathbb{S}^3, \xi)$ by contactomorphisms.

The edge $u_1v_1$ is the arc $c(t) = (2t^2-2t+1)^{-{1}/{2}}(t, 1-t, 0, 0)$, $0 \leq t \leq 1$.
Then $c'(t) =(2t^2-2t+1)^{-{3}/{2}}(1-t,-t, 0, 0 )$. Clearly, the tangent vector $c'(t) \in \xi_{c(t)}$. So, $c(t)$ is Legendrian.
  Since ${\rm Aut}(\Sigma)$ acts transitively on the edges, all the edges are Legendrian arcs.

Now, let $\Delta$ be the 2-face $u_1v_1w_1$. Let $p=p(t_1,t_2,t_3) = (t_1^2+t_2^2+t_3^2)^{-{1}/{2}}(t_1, t_2,t_3, 0)$, where
$t_1+t_2+t_3 =1$, $0 \leq t_1, t_2, t_3 \leq 1$, be a point in $\Delta$. If $p \in \Delta^{\circ}$ then
\begin{align*}
\frac{\partial p }{\partial t_1} & = (t_1^2+t_2^2+t_3^2)^{- \frac{3}{2}}(t_1^2 +t_2^2+t_3^2+ t_1t_3, t_1t_3, -(t_1^2+t_2^2), 0).
 \end{align*}
So, the tangent vector $\frac{\partial p }{\partial t_1}$ to $\Delta$ at $p$ is not in $\xi_p$. Therefore, $T_p \Delta \neq \xi_p$.
Hence, $\Delta$ is not an overtwisted disk. Since ${\rm Aut}(\Sigma)$ acts transitively on $2$-faces, no 2-face is an overtwisted disk.

Finally, $(\mathbb{S}^3, \xi)$ is tight implies that no 3-face contains an overtwisted disk. Thus, $\Sigma$ is a contact triangulation
of $(\mathbb{S}^3, \xi)$.
\end{example}

\begin{example} \label{eg:S2xS1v10}
The standard tight contact structure on $\mathbb{S}^2 \times  \mathbb{S}^1 \subset \mathbb{R}^3 \times \mathbb{S}^1$ is given by
$\text{ker}(zd\theta +xdy-ydx)$, where $(x,y,z)$ are coordinates on $\mathbb{R}^3$ and $\theta$ denotes the coordinate along $\mathbb{S}^1$-direction.

Consider the 4-dimensional simplicial complex $X= \{v_iv_{i+1}v_{i+2}v_{i+3}v_{i+4}; i \in \mathbb{Z}_{10}\}$ on the vertex set
$V = \{ v_i : i \in \mathbb{Z}_{10}\}$. Then its boundary
\begin{align} \label{eq:S21}
S^{2,1}_{10}  := \partial X & = \left \{v_iv_{i+1}v_{i+2}v_{i+4},  v_iv_{i+1}v_{i+3}v_{i+4}, v_iv_{i+2}v_{i+3}v_{i+4} : i \in \mathbb{Z}_{10} \right\}
\end{align}
triangulates $\mathbb{S}^2 \times \mathbb{S}^1$ as shown in \cite{BD08}. The complex $S^{2,1}_{10}$ was first described
by Walkup in \cite{Wa70}. Walkup also showed that any triangulation of $\mathbb{S}^2\times \mathbb{S}^1$ contains at least 10 vertices.
By Lemma \ref{tight_triangulation}, this minimal triangulation $S^{2,1}_{10}$ of $\mathbb{S}^2\times \mathbb{S}^1$ can be turned
into a 10-vertex contact triangulation of $\mathbb{S}^2 \times  \mathbb{S}^1$ with standard tight contact structure.
\end{example}

We give some triangulations of the solid torus $\mathbb{D}^2 \times \mathbb{S}^1$. Using these triangulations we shall describe
contact triangulations of some contact 3-manifolds.

\begin{example}\label{2-torus}
We know that any triangulation of the $2$-torus $S^1\times S^1$ needs at least 7 vertices and there is a unique 7-vertex
triangulation of the torus (namely, $\tau$ in Figure 3 (a)). This unique 7-vertex minimal triangulation of the torus was
first described by M\"obius in \cite{Mo}. Therefore, any triangulation of the solid torus needs at least 7 vertices. Here,
we explicitly describe smooth triangulations for the solid torus with 7 vertices. First, consider $\mathbb{S}^1 \times
\mathbb{S}^1 \subset \mathbb{C}^2$ given by $\{(z_1,z_2) \, :\,  |z_1|=1, |z_2|=1\}$ and the simple closed curves given by
\begin{subequations}
\begin{align}
C_1 & = \left\{(e^{4 \pi i t }, e^{2\pi i t}): t \in [0,1] \right \}, \\
C_2 &= \left\{(e^{ 6\pi i s}, e^{-4 \pi i s}) : s \in [0,1]   \right \}, \\
C_3 &= \left\{(e^{2\pi i \theta}, e^{ -6\pi i \theta}) : \theta \in [0,1] \right\}.
\end{align}
\end{subequations}


\begin{figure}
\centering
\subfigure[7-vertex triangulation \boldmath{$\mathcal{T}$} of 2-torus ]{
\resizebox{67mm}{!}{\label{dis_triang} 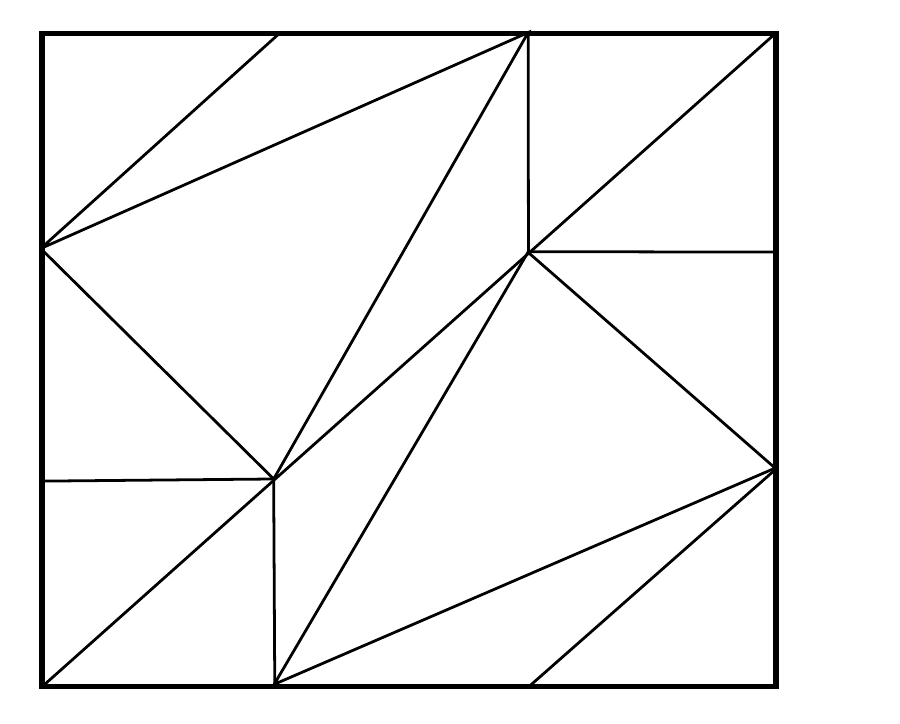 }
}
\subfigure[Three smooth curves on 2-torus]{\label{smooth_2-torus} \resizebox{66mm}{!}{
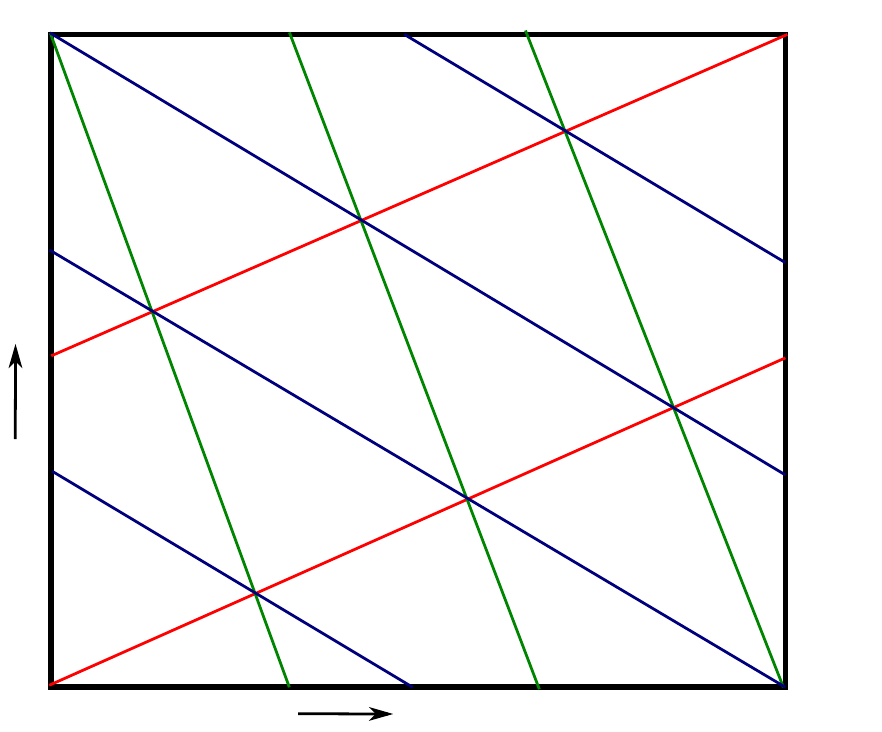
}}
\caption{Discrete and Smooth 7-vertex triangulations of 2-torus}



\end{figure}

Note that $C_1 \cap C_2 =C_2 \cap C_3 = C_1 \cap C_3 = \{u_k :=(e^{4 \pi i k/7}, e^{2 \pi i k /7}) \, : \, 0\leq k \leq 6\}$ (see Figure \ref{smooth_2-torus}). Then the seven points $u_0, u_1,\dots, u_6$ divide each curve $C_i$, $1 \leq i \leq 3$, into seven arcs. More explicitly, $C_1$ is the union of the seven arcs $u_iu_{i+1}$, $C_2$ is the union of the seven arcs $u_iu_{i+4}$ and $C_3$ is the union of the seven arcs $u_iu_{i+5}$, $0\leq i \leq 6$  (additions are modulo $7$). These 21 arcs divide the boundary torus into 14 triangular regions, namely, $u_iu_{i+1}u_{i+3}$, $u_iu_{i+2} u_{i+3}$, $0 \leq i \leq 6$. Then these 7 vertices (points) $u_0,u_1,\dots, u_6$, 21 edges (arcs) and 14 triangles (triangular regions) give the 7-vertex triangulation $\tau$ of the torus $\partial(\mathbb{D}^2 \times \mathbb{S}^1)$.

For each $i \in \{0,1, \dots, 6\}$, consider the closed curve $\ell_i :=u_iu_{i+1}u_{i + 2}u_{i}$ consisting of the three arcs $u_iu_{i+1}$ ($\subset C_1$), $u_{i+1}u_{i+2}$ ($\subset C_1$) and $u_{i+2}u_{i}$ ($\subset C_3$). The loop $\ell_i$ is a meridian in the solid torus $\mathbb{D}^2 \times \mathbb{S}^1$. Now choose seven meridional disks $\Delta_i$, $0\leq i\leq 6$, such that $\partial \Delta_i = \ell_i$ and $\Delta_i \cap \Delta_j = \ell_i \cap \ell_j $. Observe that the closure of the complement of $\cup \Delta_i$ is the union of seven balls $B_0, B_1,\dots,B_6$ such that $\partial B_i$ consists of four triangular disks $u_iu_{i+1}u_{i+3}$, $u_iu_{i+2}u_{i+3}$, $\Delta_i$ and $\Delta_{i+1}$. (First two triangular disks are on the boundary $\partial(\mathbb{D}^2 \times \mathbb{S}^1)$ and the last two disks are in the interior of the solid torus.) Let $T_1$ be the (smooth) triangulation of the solid torus whose 3-faces are $B_0, \dots, B_6$.

Similarly, we have two more triangulations which are described below in discrete representation. The smooth triangulation may be obtained by considering the curves $C_i$'s and their intersection points as described above.
\begin{subequations}
\begin{align}
T_1 & = \left\{ u_i u_{i+1}u_{i+2}u_{i+3} : i \in \mathbb{Z}_7 \right\}, \\
T_2 &= \left\{u_iu_{i+2}u_{i+3}u_{i+4} : i \in \mathbb{Z}_7   \right\}, \\
T_3 &= \left\{u_iu_{i+1}u_{i+4}u_{i+5} : i \in \mathbb{Z}_7   \right\}.
\end{align}
\end{subequations}
It is easy to see that $T_i$ gives a triangulation of $\mathbb{D}^2\times \mathbb{S}^1$  for each $i =1,2,3$. Also note that $T_1 \cap T_2 = T_1 \cap T_3 = T_2 \cap T_3 = \partial T_1 = \partial T_2 = \partial T_3 = \tau$.

For $1\leq i < j \leq 3$, let $S_{ij}$ denote the union $T_i \cup T_j$. For any vertex $u_k$, the link of $u_k $ in $S_{ij} $  is the union of the links of $u_k$ in $T_i$ and $T_j$. More formally, ${\rm lk}_{S_{ij}} (u_k) = {\rm lk}_{T_i} (u_k) \cup {\rm lk}_{T_j}(u_k)$. Since ${\rm lk}_{T_i} (u_k) \cap {\rm lk}_{T_j}(u_k) = {\rm lk}_{\tau}(u_k)$ is a cycle, it follows that ${\rm lk}_{S_{ij}} (u_k)$ is a 2-sphere. Thus, $S_{ij} $ is a closed triangulated $3$-manifold. Observe that $\alpha_1 = [u_0u_1u_6u_0]$, $\alpha_2=[ u_0 u_2u_5u_0]$ and $\alpha_3=[u_0u_3u_4u_0]$ are loops in $\tau$. Further, we have
\begin{align*}
\pi_1 (\tau, u_0) =  \langle \alpha_1, \alpha_2 \rangle = \langle \alpha_2, \alpha_3 \rangle = \langle \alpha_1, \alpha_3 \rangle.
\end{align*}
Moreover, the loop $\alpha_i $ is homotopically trivial in $T_i$ for each $i \in \{1, 2, 3\}$. Therefore, $|S_{ij}|$ (the geometric carrier of $S_{ij}$) is simply-connected for $1\leq i < j \leq 3$. Thus, $|S_{ij}|$ is homeomorphic to $\mathbb{S}^3$.

In fact, we can identify $|T_i| \cong \mathbb{D}^2 \times \mathbb{S}^1$ with the solid torus $A := \{(z_1, z_2)\in \mathbb{S}^3 \, : \, |z_1|^2\leq 1/2\}$ via the map $(z, w) \mapsto (z/\sqrt{2}, (1-|z|^2/2)^{1/2}w)$ and $|T_j| \cong  \mathbb{D}^2 \times \mathbb{S}^1$ with the solid torus $B := \{(z_1, z_2)\in \mathbb{S}^3 \, : \, |z_2|^2\leq 1/2\}$ via the map $(z, w) \mapsto ((1-|z|^2/2)^{1/2}w, z/\sqrt{2})$
to realise $|S_{ij}| =\mathbb{S}^3$.  With this notation, by a triangulation of $\mathbb{D}^2 \times \mathbb{S}^1$, we mean a triangulation of a solid 3-torus inside $\mathbb{S}^3$.
\end{example}

Consider the unknot $U= \{(0, 0, x_2, y_2) : x_2^2 +y_2^2=1\} \subset A \subset \mathbb{S}^3$. Notice that it is a positively transverse
unknot in the standard contact structure $\xi_{std}$. Let $\xi^U_{std}$ denote the contact structure obtained by Lutz twisting $\xi_{std}$ along $U$. We now prove

\begin{lemma} \label{d_3_-1}
There exists a $7$-vertex contact triangulation $S_0$ of $(\mathbb{S}^3, \xi_{std}^U)$.
\end{lemma}

We need the following two lemmas for the proof of Lemma \ref{d_3_-1}.

\begin{lemma}\label{disk-lemma}
Let $T_1$ be the triangulation of a solid torus $\mathbb{D}^2\times \mathbb{S}^1$ as in Example $\ref{2-torus}$.
Then there exists $\delta \in (0, 1)$ such that the meridional disk $D_r \times \{\theta\}$ is not contained in
any $3$-simplex of $T_1$ for any $\theta \in \mathbb{S}^1$ and $\delta < r <1$.
\end{lemma}

\begin{proof}
 For a fixed $i \in \{0, 1, \dots, 6\} $, let $B_i$ be a 3-simplex in $T_1$.
Let $D_r(t)$ denote the meridional disk $\mathbb{D}_r \times \{ e^{2\pi it}\}$.
Let $\Omega_t = D_1(t) \cap B_i $.
Then $\Omega_t$ is a single point or is given by one of the diagram shown in Figure 4.

\begin{figure}[!ht]
 \centering
 \includegraphics[scale=.5]{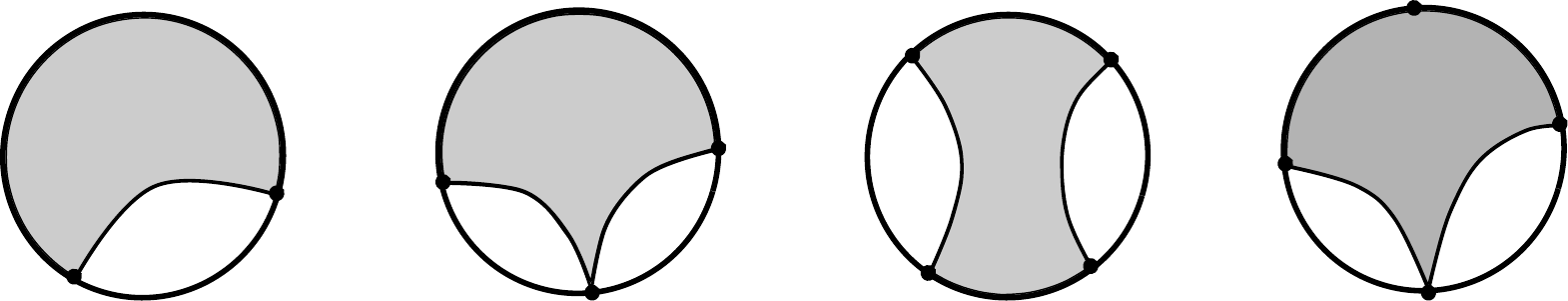}
 \caption*{Figure 4: Possible intersections of meridional disk with a tetrahedron} 
\end{figure}

Observe that $\partial \Omega_t$ intersects the interior of $D_1(t)$ nontrivially when $\Omega_t$ is not a point.
Let $A_i = (\{0 \} \times \mathbb{S}^1) \cap B_i$. Then $A_i$ is a closed
arc with length ${1}/{7} $ of $\{0\}\times \mathbb{S}^1$.

Now consider the function $f : \{0\}\times\mathbb{S}^1 \rightarrow  [0,1]$ given by
$f(e^{2 \pi it})= \sup\{r \, : \, D_r(t)\} \subset \Omega_t$. Then, by earlier observation,
$f(a) < 1$ for all $a \in \{0\} \times \mathbb{S}^1$. Since $f$ is continuous and $\{0\}\times \mathbb{S}^1$
is compact, $f$ has a maximum value, say, $\delta_i<1$. Clearly, $\delta_i = \sup\{f(a) : a \in \{0\}\times \mathbb{S}^1\} = f(p_i)$
for some $p_i\in A_i$.

Let $\delta = max \{\delta_1, \delta_2,\dots, \delta_6 \}$. Then $\delta < 1$ and for any $r \in (\delta, 1]$,
we have $D_r(t) \nsubseteq \Omega_t$.
\end{proof}

Let $(r, \phi)$ denote the polar coordinates on the disk $\mathbb{D}^2_R$ of radius $R>0$ and
$\theta $ denote coordinate along $\mathbb{S}^1$-direction in $\mathbb{D}^2 \times \mathbb{S}^1$.

\begin{lemma}\label{Lutz_twist_unknot}
Let $0< R<1$. Then there exists a contact structure $\xi^U$ on the solid torus $\mathbb{D}^2_1 \times \mathbb{S}^1$
obtained by Lutz twisting the standard contact structure $\text{ker}(d\theta +r^2 d \phi)$ along the core of solid torus
such that the following hold.
\begin{enumerate}[{\rm (i)}]
\item
$\xi^U$ is overtwisted and $\mathbb{D}^2_R \times \{ \theta \} \subset \mathbb{D}^2_1 \times \{ \theta\}$ is an overtwisted disk for each $\theta \in \mathbb{S}^1 $.
\item
There is no smaller overtwisted disk within $\mathbb{D}^2_R \times \{ \theta \}$.
\item
Near the boundary of the torus $\mathbb{D}^2_1 \times \mathbb{S}^1$ the contact structure agrees with $\text{ker}(d\theta - r^2 d\phi)$.
 \end{enumerate}
\end{lemma}

\begin{figure}[ht!]
 \centering{
 \resizebox{60mm}{!}{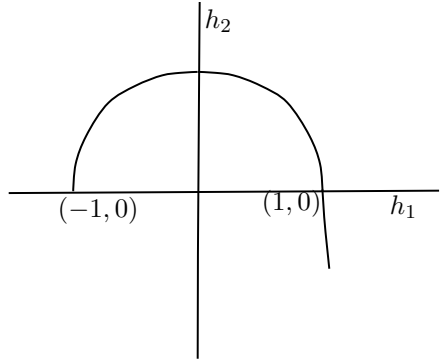}
 \caption*{Figure 5: The Choice of $(h_1(r), h_2(r)) $ for Lutz twist}\label{h1h2}
}
\end{figure}

\begin{proof}
We consider an 1-form $\alpha = h_1(r) d \theta + h_2 (r) d \phi$. It is easy to show that $\alpha$
defines a contact form if and only if the vectors $(h_1(r),h_2(r))$ and $(h'_1(r), h'_2(r))$ are linearly independent.
To achieve conditions (i) and (ii), we take $h_1(r)=-cos(\frac{\pi}{R} r)$ and $h_2(r) = r^2 sin(\frac{\pi}{R} r) $ on the interval $[0, R]$
and further extend them so that $h_1(r) = 1 $ and $h_2(r)= -r^2 $ near $r=1$. This is shown in Figure 5,
where $(h_1(r),h_2(r))$ is a parametrized curve such that the tangent $(h'_1(r), h'_2(r))$ is never parallel to the position vector at any point
on the curve.
\end{proof}

\begin{proof}[Proof of Lemma \ref{d_3_-1}]
By Lemma \ref{disk-lemma}, there is $ 0<\delta < 1$ such that no meridional disk of radius $r > \delta$ is
contained in any tetrahedron of the triangulation $T_1$.
Choose $R$ and $\epsilon > 0$ so that $\delta < R < R + \epsilon <1$.
As discussed in the Example \ref{2-torus}, we take a triangulation of $\mathbb{S}^3$ by
taking union $T_1 \cup T_2$ of triangulated solid tori. On $T_1$ we perform a Lutz twist along the core
of $T_1$ (this is unknot $U$ in $\mathbb{S}^3 $) by using Lemma \ref{Lutz_twist_unknot} with $R > \delta $.
Let $(r_i, \phi_i, \theta_i) $ denote coordinates on $T_i $ for $i=1,2 $. Then
near the boundary of the $ T_1$ the new contact structure $\xi^U $ agrees with $\text{ker} (d\theta_1 - r^2 d\phi_1)$
as given by Lemma \ref{Lutz_twist_unknot}.
The contact structure $\xi_{std} $ near the boundary of $T_2$ agrees with $\text{ker}(d\theta_2 +r^2_2 d \phi_2)$.
Use an orientation reversing diffeomorphism of $\psi: \partial T_1 \rightarrow \partial T_2$
with $\psi(\theta_1) = \phi_2$ and $\psi(\phi_1)= \theta_2$.
Note that $\psi $ preserves the contact structures near the boundary. Hence we obtain a
contact structure on $\mathbb{S}^3$. This new contact structure is obtained by Lutz twisting
$\xi_{std}$ along $U$. We denote it $\xi_{std}^U $.
We see that all overtwisted disks of $\xi_{std}^U$ lie in $T_1$ and have the radius $R$. By Lemma
\ref{Lutz_twist_unknot}, no overtwisted disk is contained in any tetrahedron.
No $2 $-face of $T_1 $ is an overtwisted disk by construction. Now, by applying a
small perturbation to all 2-faces in $T_1 \cup T_2 $, we make the edges Legendrian.
The $2$-face which are in the interior of $T_1 $ are perturbed in such a way that
they remain transverse to overtwisted disks. In particular, the new interior $2 $-faces
for $T_1$ do not coincide with any overtwisted disk. The new perturbed triangulation has tetrahedra
that slightly perturbed from their original position.
This ensures that no overtwisted disk is contained in the new tetrahedra. Let the final (after perturbations)
triangulations we get from $T_1$ and $T_2$ be $T^U_1$ and $T^U_2$ respectively.
Then $S_0 = T^U_1\cup T^U_2$ serves our purpose.
\end{proof}

For $0< \varepsilon <1/2$, let $K_{\epsilon} := \{(({{1}/{2} - \varepsilon})^{1/2}\cos 3\theta, ({{1}/{2} - \varepsilon})^{1/2}\sin 3\theta, ({{1}/{2} + \varepsilon})^{1/2} \cos 2\theta,$ $({{1}/{2} + \varepsilon})^{1/2}\sin 2\theta) \, : \, 0\leq \theta\leq 2\pi\}$. Then $K_{\varepsilon}$ lies on the torus $\{(z_1, z_2) \in \mathbb{S}^3 \, : |z_1|^2 = 1/2- \varepsilon, |z_2|^2 = 1/2 +\varepsilon\}$ and represents a trefoil knot. Moreover, $K_{\varepsilon}$
is a positively transverse knot in $(\mathbb{S}^3, \xi_{std})$. Let $\xi_{std}^{K_{\varepsilon}}$ denote the contact structure obtained by Lutz twisting the standard contact structure $\xi_{std}$ along $K_{\varepsilon}$. Here we prove

\begin{lemma} \label{d_3_+1}
There exists a $7$-vertex contact triangulation $S_{\varepsilon}$ of $(\mathbb{S}^3, \xi_{std}^{K_{\varepsilon}})$ for $0<\varepsilon<1/64$.
\end{lemma}

\begin{proof}
Let $A=\{(z_1, z_2) \in \mathbb{S}^3 : |z_1|^2 \leq 1/2\}$ and $B = \{(z_1, z_2) \in \mathbb{S}^3 : |z_2|^2 \leq 1/2\}$ be the solid tori in $\mathbb{S}^3$ as in Example \ref{2-torus}. Observe that $K_{\varepsilon}$ is inside $A$. We take $7$-vertex contact triangulations $T_1$ for $A$ and $T_2$ for $B$ as before.

Now, we perform a Lutz twist in a small tubular neighborhood $N_{\varepsilon} := \{x\in \mathbb{S}^3 : d(x, K_{\varepsilon})<\sqrt{\varepsilon}\}\equiv \mathbb{D}_{\sqrt{\varepsilon}} \times \mathbb{S}^1$ of $K_{\varepsilon}$. (Here $d$ is the standard metric on $\mathbb{S}^3$.) Observe that both the tori $A$ and $B$ intersect nontrivially with $N_{\varepsilon}$ and with the meridional disk $f(\mathbb{D}_{\sqrt{\varepsilon}} \times \{\theta \})$ for every diffeomorphism $f : \mathbb{D}_{\sqrt{\varepsilon}} \times \mathbb{S}^1 \to N_{\varepsilon}$ and $\theta \in \mathbb{S}^1$. Thus, no overtwisted disk of $\xi_{std}^{K_{\varepsilon}}$ is contained in any tetrahedron of $T_1\cup T_2$.

Now, we apply small perturbations to the $2$-faces in $T_1 \cup T_2$ so that the edges become Legendrian. Since $\varepsilon$ is small (compared to 1), we perturb the 2-faces so that no 2-face is an overtwisted disk and no tetrahedron contains an overtwisted disk. Let the final (after perturbations) triangulations we get from $T_1$ and $T_2$ be denoted by $T^{\varepsilon}_1$ and $T^{\varepsilon}_2$ respectively.
Then $S_{\varepsilon} = T^{\varepsilon}_1\cup T^{\varepsilon}_2$ serves our purpose.
\end{proof}

From Example \ref{eg:S2xS1v10}, we know that vertex minimal contact triangulation exists for $\mathbb{S}^2 \times  \mathbb{S}^1$ with standard tight contact structure.
We now show that $S^{2,1}_{10}$ given by \eqref{eq:S21} can be turned into a contact triangulation of $\mathbb{S}^2 \times  \mathbb{S}^1$ with an overtwisted contact structure.

\begin{example}
Let $S^{2,1}_{10}$ be as in \eqref{eq:S21}. Consider the following subcomplexes of $S^{2,1}_{10}$.
\begin{align}
 T_4 &= \left \{v_{2i}v_{2i+1}v_{2i+2}v_{2i+4}, v_{2i+1}v_{2i+2}v_{2i+4}v_{2i+5}, v_{2i}v_{2i+2}v_{2i+3}v_{2i+4} :
 i \in \mathbb{Z}_{10} \right \}, \nonumber \\
T_5 &= \left \{v_{2i+1}v_{2i+2}v_{2i+3}v_{2i+5}, v_{2i}v_{2i+1}v_{2i+3}v_{2i4}, v_{2i+1}v_{2i+3}v_{2i+4}v_{2i+5} :
i \in \mathbb{Z}_{10}    \right \}.
\end{align}
Notice that $S = T_4 \cup T_5$ and $T_4 \cap T_5 = \partial T_4 = \partial T_5 = \tau_1$, where $\tau_1$ is as in Figure 6.

\setlength{\unitlength}{3mm}
\begin{picture}(30,15)(10,3)

\thicklines

\put(22,16){\line(1,0){15}} \put(22,8){\line(1,0){15}}

\put(22,8){\line(0,1){8}} \put(37,8){\line(0,1){8}}

\thinlines

\put(22,12){\line(21,0){15}}

\put(25,8){\line(0,1){8}} \put(28,8){\line(0,1){8}}
\put(31,8){\line(0,1){8}} \put(34,8){\line(0,1){8}}

\put(25,8){\line(-3,4){3}}
\put(28,8){\line(-3,4){6}} \put(31,8){\line(-3,4){6}}
\put(34,8){\line(-3,4){6}} \put(37,8){\line(-3,4){6}}
\put(37,12){\line(-3,4){3}}

\put(22,6.8){\mbox{$v_0$}} \put(25,6.8){\mbox{$v_4$}}
 \put(28,6.8){\mbox{$v_8$}} \put(31,6.8){\mbox{$v_2$}}
 \put(34,6.8){\mbox{$v_6$}} \put(37,6.8){\mbox{$v_0$}}

\put(22.2,12.3){\mbox{$v_1$}} \put(25.2,12.3){\mbox{$v_5$}}
 \put(28.2,12.3){\mbox{$v_9$}} \put(31.2,12.3){\mbox{$v_3$}}
 \put(34.2,12.3){\mbox{$v_7$}} \put(37.2,12.3){\mbox{$v_1$}}

 \put(22,16.3){\mbox{$v_2$}} \put(25,16.3){\mbox{$v_6$}}
 \put(28,16.3){\mbox{$v_0$}} \put(31,16.3){\mbox{$v_4$}}
 \put(34,16.3){\mbox{$v_8$}} \put(37,16.3){\mbox{$v_2$}}

\put(20,4){\mbox{Figure\,6\,: {\bf 10-vertex torus \boldmath{$\mathcal{T}_1$} }}}

\end{picture}

The automorphism group of the simplicial complex $S^{2,1}_{10}$ is generated by $\alpha$, $\beta$ and $\gamma$, where $\alpha(v_i)= v_{i+2}$, $\beta (v_i)=v_{-i}$ and $\gamma(v_i)=v_{i+5}$, $i\in \mathbb{Z}_{10}$. Then $\gamma (T_4) =T_5$ and $\gamma (T_5) = T_4$. Consider the triangulated 3-balls
\begin{align}
B_1 &= \left \{v_0v_1v_2v_4, v_1v_2v_4v_5, v_2v_4v_5v_6 \right\}, \nonumber \\
B_2 &= \left \{v_2v_3v_4v_6, v_3v_4v_6v_7, v_4v_6v_7v_8 \right\}, \nonumber \\
B_3 &= \left \{v_5v_6v_8v_9, v_6v_8v_9v_0, v_6v_7v_8v_0, v_7v_8v_0v_1 \right\}, \nonumber  \\
B_4 &= \left \{v_8v_0v_1v_2, v_8v_9v_0v_2, v_9v_0v_2v_3, v_0v_2v_3v_4  \right\}.
\end{align}
Then $B_1 \cap B_2 = \{v_2v_4v_6, v_4v_5v_6 \}$ and $B_3 \cap B_4 = \{v_8v_0v_1, v_8v_9v_0\}$. Let $B_{12} = B_1 \cup B_2 $ and $B_{34} = B_3 \cup B_4$. Observe that $B_{12}$ and $B_{34}$ are triangulations of the 3-ball. Also, $B_{12} \cap B_{34} = \{v_5v_6v_8, v_6v_7v_8\} \cup \{v_0v_1v_2, v_0v_2v_4, v_2v_3v_4\}$ is a triangulation of the disjoint union of two 2-disks. Since $T_4 = B_{12} \cup B_{34}$ and $\partial T_4 = \tau_1$ is the torus, it follows that $|T_4|$ is $\mathbb{D}^2 \times \mathbb{S}^1$. As $\gamma(T_4) =T_5$, it follows that $T_5$ also triangulates $\mathbb{D}^2 \times \mathbb{S}^1$. Therefore, $S^{2,1}_{10}$ is the union of two triangulated solid tori and it triangulates $\mathbb{S}^2 \times \mathbb{S}^1$. Observe that the edges $v_0v_2, v_2v_4, v_4v_6, v_6v_8, v_8v_0$ are in the interior of $T_4$ and the edges $v_1v_3, v_3v_5,  v_5v_7, v_7v_9, v_1v_9$ are in the interior of $T_5$.

Assume, without loss, that $|T_4|= D^2_+\times\mathbb{S}^1 = \{((x_1, x_2, x_3), z) \, : \, (x_1, x_2, x_3)\in \mathbb{S}^2, x_3\geq 0, z\in\mathbb{S}^1\}$ and $|T_5|= D^2_{-}\times\mathbb{S}^1 = \{((x_1, x_2, x_3), z) \, : \, (x_1, x_2, x_3)\in \mathbb{S}^2, x_3\leq 0, z\in\mathbb{S}^1\}$. Let $\xi_{std}$ be the standard contact structure on $\mathbb{S}^2\times \mathbb{S}^1$ defined in Example \ref{eg:S2xS1v10}.

For $0< \varepsilon <1/16$, let $L_{\epsilon} := \{(((1 - \varepsilon)^{1/2}\cos 3\theta, (1 - \varepsilon)^{1/2}\sin 3\theta, \varepsilon^{1/2}), ((1 + \varepsilon)^{1/2} \cos 2\theta,$ $(1 + \varepsilon)^{1/2}\sin 2\theta) \, : \, 0\leq \theta\leq 2\pi\}$. Then $L_{\varepsilon}$ represents a trefoil knot and lies inside $D^2_+\times\mathbb{S}^1$.

As in the proof of Lemma \ref{d_3_+1}, we perform a Lutz twist in the small tubular neighborhood $W_{\varepsilon} := \{(x,z)\in \mathbb{S}^2\times\mathbb{S}^1 : d((x,z), L_{\varepsilon})<2\sqrt{\varepsilon}\}\equiv \mathbb{D}_{2\sqrt{\varepsilon}} \times \mathbb{S}^1$ of $L_{\varepsilon}$. (Here $d((x,z), (y,w))= (d_2(x,y)^2+ d_1(z, w)^2)^{1/2}$, where $d_n$ is the Euclidean metric on $\mathbb{S}^n$.)
Then both the tori $|T_4|$ and $|T_5|$ intersect nontrivially with $W_{\varepsilon}$.
Let $\xi^{L_{\varepsilon}}$ be the contact structure on $\mathbb{S}^2 \times \mathbb{S}^1$  obtained by this Lutz twisting.

By the similar argument as in the proof of Lemma \ref{d_3_+1}, using small perturbations of the 2-faces of $S^{2,1}_{10}$, we get a 10-vertex contact triangulation $S^{2,1}_{\varepsilon}$ of $(\mathbb{S}^2 \times \mathbb{S}^1, \xi^{L_{\varepsilon}})$.
\end{example}

\section{Proofs}

First consider the following definition (\cite[Definition 1.3]{BD08}).

\begin{definition} \label{def:CSum}
Let $X_1$ and $X_2$ be triangulations of two closed $d$-manifolds. Assume, without loss, that they have no common vertices. For $1\leq i\leq 2$, let $\sigma_i\in X_i$ be a $d$-simplex. Let $\psi : \sigma_1 \to \sigma_2$ be a
bijection.
Let $(X_1\sqcup X_2)^{\psi}$ denote the simplicial complex obtained from $X_1\sqcup X_2 \setminus \{\sigma_1, \sigma_2\}$ by identifying $x$ with $\psi(x)$ for each $x\in \sigma_1$.
Then, $(X_1\sqcup X_2)^{\psi}$ is a triangulation of the connected sum of $|X_1|$ and $|X_2|$ (taken with appropriate orientations) and called an {\em
elementary connected sum} of $X_1$  and $X_2$, and is denoted by
$X_1 \#_{\psi} X_2$ (or simply by $X_1\# X_2$).
\end{definition}

\begin{proof}[Proof of Theorem \ref{mainthm-1}]
Observe that that the connected sum $(\mathbb{S}^3, \xi)\# (\mathbb{S}^3, \xi_{std})$ is contactomorphic to  $(\mathbb{S}^3, \xi)$. The contact structure $\xi_+$ is obtained by performing a Lutz twist along the right handed trefoil $K$. In particular, we think of $K \subset A \subset (\mathbb{S}^3, \xi_{std})$ as discussed in the proof of Lemma \ref{d_3_+1}. Recall that $T_1$ triangulates $A$. Now, consider the 7-vertex triangulation $S_{12}=T_1 \cup T_2$ of $\mathbb{S}^3$ as given in Example \ref{2-torus}.
Assume, without loss, that $S_{12}$ and $\Sigma$ have no common vertex.
Let $\sigma \in T_2$ and $\mu \in \Sigma$ be two 3-simplices and let $\psi : \mu\to \sigma$ be a bijection. Let $\Sigma\# S_{12}= \Sigma\#_{\psi}  T_2$ be the elementary connected sum. Since the dimension is 3, $\Sigma\# S_{12}$ triangulates the connected sum $|\Sigma|\# |S_{12}|$.
We may assume that $\Sigma \# S_{12}$ is smooth by applying small enough perturbation wherever it is necessary. Now observe that the knot $K \subset A$ is disjoint from $\sigma \in T_2 \subset S_{12}$. We perform a Lutz twist on the contact structure $\xi$ along $K$ as shown in the proof of Lemma \ref{d_3_+1}. We obtain the new contact structure $\xi_+$. Notice that $\xi_+$ restricted to $|\Sigma \setminus \mu|$ ($=|\Sigma| \setminus \mu^{\circ}$) agrees with $\xi$ as the Lutz twist along $K$ has changed the contact structure only in the neighborhood $N \subset A$ of $K$. This implies that edges in $\Sigma \setminus \mu$ are Legendrian with respect to $\xi_+$. No 2-faces are overtwisted disks and $\xi_+$ restricted to interior of each tetrahedron in $\Sigma \setminus \mu$ is tight. We may need to perturb the 2-faces of $\Sigma \# S_{12}$ to make {\em all} the edges in $\Sigma \# S_{12}$ Legendrian. Note that no overtwisted disk lies in any tetrahedron from $T_1$ and $T_2 \setminus \sigma$ as discussed in the proof of Lemma \ref{d_3_+1}. Therefore, $\xi_+$ restricted to the interiors of {\em all} the tetrahedra is tight and no 2-face is an overtwisted disk. Hence, $\Sigma \# S_{12}$  is a contact triangulation for $(\mathbb{S}^3, \xi_+)$. The number of vertices in the new triangulation is $f_0(\Sigma)+7-4 = f_0(\Sigma)+3$.

The argument to obtain a contact triangulation for $(\mathbb{S}^3, \xi_-)$ is similar. The role of the right handed trefoil is played by the unknot $U$ in the above argument.
\end{proof}

\begin{proof}[Proof of Corollary \ref{cor-1}]
We know that overtwisted contact structures on $\mathbb{S}^3$ are classified completely by their $d^3$ invariant following Eliashberg \cite{El89}. Assume that $d^3(\xi_{ot}, \xi_{std}) = n >0$.

If $n=1$ then the result follows from Lemma \ref{d_3_+1}. Let $S_{\varepsilon}$, $T^{\varepsilon}_1$, $T^{\varepsilon}_2$ be as in the proof of Lemma \ref{d_3_+1}.

Now, assume that $n\geq 2$ and the result is true for $n-1$. Let $S_{n-1}$ be a $(3(n-1) + 4)$-vertex contact
triangulation of $(\mathbb{S}^3, \xi^{\prime})$, where $d^3(\xi^{\prime}, \xi_{std}) = n-1$. Let $\sigma$ be
a 3-simplex in $S_{n-1}$ and $\beta$ be a 3-simplex in $T^{\varepsilon}_2\subseteq S_{\varepsilon}$. (Observe that the copy of $K_{\varepsilon}$ inside
$S_{\varepsilon}$ is disjoint from $\beta$.) Let $S_n$ be the connected sum of $S_{n-1}$ and $S_{\varepsilon}$ obtained as the union
of $S_{n-1}\setminus\{\alpha\}$ and $S_{\varepsilon}\setminus\{\beta\}$. (A little perturbation makes $S_n$ a smooth triangulation.)
Then the new trefoil $K_{\varepsilon}$ is unlinked with earlier trefoils (in the inductive construction). (Also observe that we can
inductively assume that $\alpha$ is disjoint from all the $n-1$ copies of trefoils in $S_{n-1}$. Clearly, $S_n$ has $3n+1 + 7 -4 = 3n+4$ vertices.

Then $S_n$ is a contact triangulation of $(\mathbb{S}^3, \xi^{\prime})\# (\mathbb{S}^3, \xi_{std}^{K_{\varepsilon}})$. From the construction it follows that $(\mathbb{S}^3, \xi^{\prime}) \# (\mathbb{S}^3, \xi_{std}^{K_{\varepsilon}})$ can be considered as obtained from $(\mathbb{S}^3, \xi^{\prime})$ by performing a Lutz twist along a copy of trefoil $K_{\varepsilon}$ which is unlinked with the earlier trefoils in $S_{n-1}$. So, $(\mathbb{S}^3, \xi^{\prime}) \# (\mathbb{S}^3, \xi_{std}^{K_{\varepsilon}})= (\mathbb{S}^3, \xi)$ for some $\xi$ with $d^3(\xi, \xi_{std})=n$. Now using the uniqueness (up to contact isotopy) of overtwisted contact structure with $d^3(\xi, \xi_{std})= n$, we see that $\xi$ is isotopic to $\xi_{ot}$. Thus, $S_n$ is a contact triangulation of $(\mathbb{S}^3, \xi_{ot})$. The result (for $n>0$) now follows by induction.

Now, assume that  $d^3(\xi_{ot}, \xi_{std}) = n < 0$. In this case, we prove the result by the same method of induction as above by replacing $K$ by unknot $U$, $S_{\varepsilon}$ by $S_0$. (Namely, the result is true for $n=-1$ by Lemma \ref{d_3_-1}. We then take successive connected sums of copies of $S_0$'s. We get $(3|n|+4)$-vertex contact triangulation of $(\mathbb{S}^3, \xi_{ot})$.)

Now assume that $d^3(\xi_{ot}, \xi_{std}) = n = 0$. In this case, the triangulation we get is from $S_0\#S_{\varepsilon}$ (by the same method as above). Thus, we have a triangulation with $(7+7-4) =10$ vertices of $(\mathbb{S}^3, \xi_{ot})$ when $d^3(\xi_{ot}, \xi_{std})=0$.
\end{proof}

\begin{proof}[Proof of Theorem \ref{mainthm-2}]
The proof is very similar to that of Theorem \ref{mainthm-1}. We take connected sum of
triangulations $\mathcal{M} $ with $(T_1 \cup T_2)$ the 7-vertex triangulation of
$(\mathbb{S}^3, \xi_{std})$. Then we perform Lutz twist along positively transverse
unknot which forms the core of $T_1 $ (trefoil which lies on the common boundary $T_1 \cap T_2$) to get $\xi_- $ (respectively $\xi_+ $).
\end{proof}

\begin{proof}[Proof of Corollary \ref{cor-2}]
We start with the given contact triangulation $\mathcal{M} $ of $(M, \xi) $. We need to perform
Lutz twist along positively transverse unknot or right-handed trefoil (both viewed as embedded
in a 3-ball inside $M$). We know that Lutz twisting along both the knots do not change the $d^2 $
invariant as both of them are null-homologus. Thus, the contact structures obtained by performing
Lutz twisting along unknot or trefoil will be classified completely by their relative $d^3$ invariant.
This follows from Eliashberg's result.

Just as in the proof of Corollary \ref{cor-1}, we observe that we need to add $3 $ vertices every time we
perform a Lutz twist along an unlinked trefoil (respectively unknot) Thus, by induction we get a
contact triangulation of $(M, \xi')$, starting from $\mathcal{M}$, with $f_0 (\mathcal{M}) + 3 |n| $
vertices if $n \neq 0 $. As before, we need to perform a Lutz twist along trefoil and then along an
unlinked unknot. Thus, we get a contact triangulation of $(M, \xi')$ with $(f_0(\mathcal{M})+6) $ vertices,
where $d^2(\xi', \xi)=0$ and $d^3(\xi', \xi)=0$.
\end{proof}

The above result says that the vertices in any minimal contact triangulations grow at most linearly with respect to the $d^3 (\xi', \xi)$ with $d^2(\xi', \xi)=0$.  A natural question at this point to ask is that under the same hypothesis, is this bound optimal?
We believe that this is the case. Hence the we would like to propose the following conjecture.

\begin{conjecture}
Let $\xi'$ be a contact structure on $M$ with $d^2(\xi', \xi)=0$ and $d^3(\xi', \xi) = n$. Then for any minimal
contact triangulation $\Sigma$ for $(M, \xi')$, $f_0(\Sigma)= O(n)$.
\end{conjecture}

The above conjecture may be proved using the following statement which is also a conjecture
we would like to propose.

\begin{conjecture}
Any minimal contact triangulation for an overtwisted contact structure  on $\mathbb{S}^3$ must have at
least $7$ vertices. Equivalently,  if there is a contact triangulation for $(\mathbb{S}^3, \xi) $ with $6$ or less vertices
then $\xi $ is isotopic to $\xi_{std} $.
\end{conjecture}

\section{3-Torus}\label{3-torus}

In this section, we give contact triangulations for a certain family of overtwisted contact structures on
3-torus obtained by Lutz twisting along a fixed knot which is not  null-homologous. We show that the number
of vertices in minimal triangulations grows at most cubically in terms of relative $d^2$ invariant. We take the 3-torus
$T^3$ as $\mathbb{R}^3 / \mathbb{Z}^3$. In particular, we glue the opposite faces of a 3-dimensional cube
$[0,1]^3 $ to obtain the 3-torus. Let $(x,y,z)$ denote a coordinate system on $\mathbb{R}^3$. We consider
the induced metric on $T^3$ as a quotient of 3-dimensional Euclidean space. In the following examples,
we will perform a Lutz twist along the knot $K=\{({1}/{2}, {1}/{2}, z) : z \in [0,1]\}$. We will produce
a family of contact structures $\xi_n $ on $T^3 $ inductively in which there is a nested sequence of $n $
overtwisted disks centered at $({1}/{2}, {1}/{2}, z) $ for $z \in [0,1] $.

\smallskip

First consider the 1-form $\alpha_0 = \cos (2 \pi x) dz+ \sin (2\pi x) dy$ on $T^3$.
Take a disk $D_{r_0} \subset [0,1]^2$ of radius $r_0\in ({1}/{4}, {1}/{2})$ and center
$({1}/{2}, {1}/{2})$. Let $(r, \phi)$ denote the polar coordinates on $[0,1]^2$ with the point
$({1}/{2}, {1}/{2}) $ as the center. Thus, we have the coordinate system $(r, \phi, z)$ on $T^3$.
Note that the point $(x,y,z)$ has new coordinates $(r, \phi, z)$ where $r= ({(x-{1}/{2})^2 + (y-{1}/{2})^2})^{1/2}$ and
$\phi$ is the angle made by the position vector $(x-{1}/{2}, y-{1}/{2})$ with line segment $y ={1}/{2}$ in $[0,1]^2$.
Therefore, $x= r \cos \phi + {1}/{2}$ and $y=r \sin \phi + {1}/{2}$.
We prove a useful lemma below.

\begin{lemma}
Let $\epsilon >0$ be such that $r_0< r_0+ \epsilon < {1}/{2} $. Let $(r, \phi)$ denote the polar coordinates
on $[0, 1]^2$. Then there exists a contact form $\beta = h_1(r, \phi) dz + h_2(r, \phi) d\phi + h_3(r, \phi) dr$ on $D_{r_0 + \epsilon}
\times \mathbb{S}^1$ such that the following hold.
\begin{enumerate}[{\rm (i)}]
\item $\beta$ agrees with $dz + r^2 d \phi $ on $ D_{r_0} \times \mathbb{S}^1$.
\item Near the boundary of $D_{r_0 + \epsilon} \times \mathbb{S}^1$, $\beta $ agrees
with the contact form $\alpha_0=\cos (2 \pi x) dz + \sin(2 \pi x) dy $ when expressed in the coordinates $( r, \phi, z)$.
\end{enumerate}
\end{lemma}

\begin{proof}
Let $\beta= h_1(r, \phi) dz + h_2(r, \phi) d\phi + h_3(r, \phi) dr$ for smooth functions $h_1, h_2 $ and $h_3$.
Then the contact condition  $\beta \wedge d \beta \neq 0$ holds if and only if
\begin{equation}\label{cc}
(h_1 \partial_r h_2 - h_2 \partial_r h_1) - (h_1 \partial_{\phi}h_3 -h_3 \partial_{\phi} h_1) \neq 0,
\end{equation}
where $\partial_r$ and $\partial_{\phi}$ denote the partial derivatives with respect to $r$ and $\phi$ respectively. We view the triple $(h_1(r, \phi),h_2(r, \phi), h_3(r, \phi))$ as a vector in $\mathbb{R}^3$. Observe that the vector $(0,0,h_1 \partial_r h_2 - h_2 \partial_r h_1)$ is the cross product of vectors $(h_1,h_2, 0)$
and $(\partial_r h_1, \partial_r h_2, 0)$. Similarly, the vector $(0, h_1 \partial_{\phi}h_3 -h_3 \partial_{\phi}h_1, 0)$ is the cross product of the two vectors $(h_1,0,h_3)$ and $(\partial_{\phi}h_1, 0, \partial_{\phi}h_3)$. Now, consider the vector $v:=(0,h_1 \partial_{\phi}h_3-h_3 \partial_{\phi}h_1, h_1 \partial_r h_2 - h_2 \partial_r h_1 )$. Then  \eqref{cc} is equivalent to the condition that  $v\neq \lambda(0,1,1) $ for any $\lambda \in \mathbb{R} $ (i.e., $v$ is non-zero and never parallel to $(0,1,1)$). Note that $\mathbb{R}^3 -\{\lambda (0,1,1) : \lambda \in \mathbb{R} \}$ is connected. Therefore, it is possible to find a triple $(h_1, h_2,h_3)$ such that $(h_1, h_2, h_3) = (1, r^2, 0)$ on $D_{r_0} \times \mathbb{S}^1$ and $h_1 = \cos 2\pi x $, $h_2 = (x-{1}/{2}) \sin 2 \pi x$ and $h_3 = \sin \phi \cdot \sin(2 \pi x)$ near the boundary of $D_{r_0 + \epsilon} \times \mathbb{S}^1 $ where $x = r \cos \phi + {1}/{2}$.
\end{proof}

By the above lemma, we see that $\beta$ and $\alpha_0$ agree near the boundary of $D_{r_0 + \epsilon} \times \mathbb{S}^1$. Thus, we have a contact form $\alpha$ on $T^3$ defined by $\beta $ on $D_{r_0 + \epsilon} \times \mathbb{S}^1$ and by $\alpha_0$ outside of it. Let $ \xi_0$ denote the $\text{ker} (\alpha)$. Then the knot $K= \{({1}/{2}, {1}/{2}, z)\}$ is a positively transverse knot in $(T^3, \xi_0)$. Notice that the contact structure on the neighborhood $D_{r_0} \times \mathbb{S}^1$ of $K$ is,
by virtue of construction, in the standard form. We will now perform a Lutz twist along $K$ on the neighborhood $D_{r_0} \times \mathbb{S}^1$ as follows. We choose functions
$h_1(r)$ and $h_2 (r)$ such that $h_1(r)=-1 $ and $h_2(r)=-r^2 $ for $r \leq {r_0}/2 $ and $(h_1(r), h_2(r))=(1, r^2)$ near $r=r_0$. Thus, the new contact structure $\xi_1 = \xi_0^K$ has overtwisted disks with radius $r \in (r_0/2, r_0)$.

We note that $d^2(\xi_1, \xi_0)= - PD([K])$. We also note that the knot $K$ is a
negatively transverse in $(T^3, \xi_1)$. Again, we perform a Lutz twist along $K$ in the neighborhood $D_{r_1} \times \mathbb{S}^1 $, where $r_1 = {r_0}/{2}$ by
choosing pair of functions $(h_1, h_2)$ which is equal to $(1, r^2) $ for $r \leq r_0/3 $ and $(-1, -r^2)$ near $r= r_0/2$. We denote the new overtwisted contact structure by $\xi_2$. For each slice $[0,1] \times [0,1] \times \{z \} $, there are now two nested overtwisted disks with the smallest overtwisted disk having radius $r \in (r_0/3, r_0/2)$. By additivity of $d^2$ invariant, we see that $d^2(\xi_2, \xi_0)= -2 PD([K])$.

To obtain $\xi_3 =\xi_2^K$, we perform a Lutz twist on the neighborhood $D_{r_0/3} \times \mathbb{S}^1$ of $K$ by using the pair of functions $(h_1, h_2)$ such that the pair agrees with $(-1, -r^2)$ for $r \leq r_0/4$ and equals $(1,r^2) $ near $r= r_0/3$. Notice that we, now, have three nested overtwisted disks with the smallest one
having radius $r \in (r_0/4, r_0/3)$. Also, observe that $d^2(\xi_3, \xi_0)=-3 PD([K])$.

Inductively, we obtain overtwisted contact structures $\xi_n$ from $\xi_{n-1}$ on $T^3$  by performing Lutz twist along $K$ in the neighborhood $D_{r_{n-1}} \times \mathbb{S}^1$, where $r_{n-1} = {r_0}/{n}$. By construction, $\xi_n$ has $n$
overtwisted disks in each section $[0,1] \times [0,1] \times \{z\}$ for $z \in [0,1]$. The smallest overtwisted disk in $[0, 1] \times [0, 1] \times \{z\}$ has the radius $r \in ({r_0}/(n+1), r_0/n)$. We see that $d^2(\xi_n, \xi_0)= - n PD([K])$.
Thus, the contact structure $\xi_n$ and $\xi_m$ are not isotopic for $n \neq m$.

Now, we describe contact triangulations for each of the contact structure on $T^3$. By construction,
these triangulations are close to minimal ones. There are two key ideas involved in the construction.
Firstly, we take the cube $[0,1]^3$ and subdivide it into cubes of smaller size in an efficient manner.
Secondly, we triangulate the smaller cubes in such a way that we get a triangulation of the $[0,1]^3$
which respects the identification on the boundary. Thus, we get a triangulation of $T^3$. While
constructing these triangulations, we make sure that no tetrahedron contains an overtwisted disk by
controlling the diameter of each tetrahedron.

Now, we describe a contact triangulation of $(T^3, \xi_1)$. Firstly, we describe
the following triangulations of the cube which will be used as basic blocks. We give
four different type of triangulations on the cube $[0, {1}/{3}]^3$. We will use them to
triangulate the cube $[0,1]^3$.

\medskip

\noindent \textbf{Type A\,:} We subdivide the cube $[0, {1}/{3}]^3$ into five tetrahedra (it gives the stacked ball structure on $[0, {1}/{3}]^3$) in the following ways (see Figure 7)
\begin{subequations}
\begin{align}\label{A0A1}
 A_0 &= \{a_0a_1a_2a_4, a_1a_2a_3a_7, a_1a_4a_5a_7, a_2a_4a_6a_7, a_1a_2a_4a_7\}, \\
A_1 &= \{ a_0a_1a_3a_5, a_0a_2a_3a_6, a_3a_5a_6a_7, a_0a_4a_5a_6, a_0a_3a_5a_6\}.
\end{align}
\end{subequations}


\setlength{\unitlength}{3.2mm}

\begin{picture}(44,17)(0,0)

\thicklines

\put(4,3){\line(1,0){8}} \put(4,3){\line(0,1){8}}
\put(4,11){\line(1,0){8}} \put(12,3){\line(0,1){8}}
\put(2,15){\line(1,0){8}} \put(2,7){\line(0,1){8}}
\put(2,7){\line(1,0){1.5}} \put(4.4,7){\line(1,0){3}}
\put(8.5,7){\line(1,0){1.5}} \put(10,7){\line(0,1){3.5}}
\put(10,11.5){\line(0,1){3.5}} \put(4,3){\line(-1,2){2}}
\put(4,11){\line(-1,2){2}} \put(12,3){\line(-1,2){2}}
\put(12,11){\line(-1,2){2}}

\thinlines

 \put(12,3){\line(-5,2){7.5}} \put(2,7){\line(5,-2){1.5}}

\put(4,11){\line(3,2){6}}
\put(12,3){\line(-1,6){1.25}} \put(10,15){\line(1,-6){0.6}}

\put(2,7){\line(1,1){1.6}} \put(10,15){\line(-1,-1){3.7}}
\put(4.2,9.06){\mbox{$\cdot$}} \put(4.4,9.25){\mbox{$\cdot$}}
\put(5.15,10){\mbox{$\cdot$}} \put(5.5,10.35){\mbox{$\cdot$}}

\put(12,3){\line(-1,1){8}}

\put(2,7){\line(1,2){2}}

\put(2.2,3.1){\mbox{$a_{0}$}} \put(12.5,3.1){\mbox{$a_{1}$}}
\put(0.8,6){\mbox{$a_{2}$}} \put(8.5,7.5){\mbox{$a_{3}$}}
\put(2.3,11.2){\mbox{$a_{4}$}} \put(12.5,10){\mbox{$a_{5}$}}
\put(0.5,14.4){\mbox{$a_{6}$}} \put(11,14.4){\mbox{$a_{7}$}}

\put(5,1.3){\small \mbox{\boldmath{$A_0$}}}


\thicklines

\put(18,3){\line(1,0){8}} \put(18,3){\line(0,1){8}}
\put(18,11){\line(1,0){8}} \put(26,3){\line(0,1){8}}
\put(16,15){\line(1,0){8}} \put(16,7){\line(0,1){8}}
\put(16,7){\line(1,0){1}} \put(18.4,7){\line(1,0){3}}
\put(22.5,7){\line(1,0){1.5}} \put(24,7){\line(0,1){1.5}}
\put(24,9.5){\line(0,1){1.1}}
\put(24,12.3){\line(0,1){2.7}} \put(18,3){\line(-1,2){2}}
\put(18,11){\line(-1,2){2}} \put(26,3){\line(-1,2){2}}
\put(26,11){\line(-1,2){2}} 

 \put(18,3){\line(3,2){6}}

 \put(26,11){\line(-5,2){10}} 

  \put(23.3,7.19){\mbox{$\cdot$}} \put(23,7.49){\mbox{$\cdot$}}

\thinlines

 \put(24,7){\line(1,2){2}}
\put(16,15){\line(1,-1){3.5}} 
\put(18,3){\line(1,1){8}}
\put(22.6,8.4){\line(-1,1){2}}


\put(18,3){\line(-1,6){2}}

\put(16.2,3.1){\mbox{$a_{0}$}} \put(26.5,3.1){\mbox{$a_{1}$}}
\put(14.8,6){\mbox{$a_{2}$}} \put(24.5,6.5){\mbox{$a_{3}$}}
\put(18.4,10.1){\mbox{$a_{4}$}} \put(26.5,10){\mbox{$a_{5}$}}
\put(14.5,14.4){\mbox{$a_{6}$}} \put(25,14.4){\mbox{$a_{7}$}}

\put(22,1.3){\small \mbox{\boldmath{$A_1$}}}


\thicklines

\put(32,3){\line(1,0){8}} \put(32,3){\line(0,1){8}}
\put(32,11){\line(1,0){8}} \put(40,3){\line(0,1){8}}
\put(30,15){\line(1,0){8}} \put(30,7){\line(0,1){8}}
\put(30,7){\line(1,0){1.5}} \put(32.4,7){\line(1,0){3}}
\put(36.5,7){\line(1,0){1.5}} \put(38,7){\line(0,1){3.5}}
\put(38,11.5){\line(0,1){3.5}} \put(32,3){\line(-1,2){2}}
\put(32,11){\line(-1,2){2}} \put(40,3){\line(-1,2){2}}
\put(40,11){\line(-1,2){2}}

\thinlines

\put(40,3){\line(-5,2){7.5}} \put(30,7){\line(5,-2){1.5}}

\put(32,11){\line(3,2){6}}
\put(40,3){\line(-1,6){1.25}} \put(38,15){\line(1,-6){0.6}}
\put(38,7){\line(-1,1){3.6}} \put(30,15){\line(1,-1){2.8}}
\put(33.5,11.1){\mbox{$\cdot$}}

\put(40,3){\line(-1,1){8}}

\put(30,7){\line(1,2){2}}

\put(30.2,3.1){\mbox{$a_{0}$}} \put(40.5,3.1){\mbox{$a_{1}$}}
\put(28.8,6){\mbox{$a_{2}$}} \put(36.7,8.5){\mbox{$a_{3}$}}
\put(30.3,11.2){\mbox{$a_{4}$}} \put(40.5,10){\mbox{$a_{5}$}}
\put(28.5,14.4){\mbox{$a_{6}$}} \put(39,14.4){\mbox{$a_{7}$}}

\put(35,1.3){\small \mbox{\boldmath{$C_{05}$}}}

\put(11,0.5){\mbox{}\label{A-type_cubes}}

\end{picture}

\setlength{\unitlength}{3.2mm}
\begin{picture}(44,17)(0,-1)

\thicklines

\put(4,3){\line(1,0){8}} \put(4,3){\line(0,1){8}}
\put(4,11){\line(1,0){8}} \put(12,3){\line(0,1){8}}
\put(2,15){\line(1,0){8}} \put(2,7){\line(0,1){8}}
\put(2,7){\line(1,0){1.5}} \put(4.4,7){\line(1,0){3}}
\put(8.5,7){\line(1,0){1.5}} \put(10,7){\line(0,1){3.5}}
\put(10,15){\line(0,-1){2.7}} \put(4,3){\line(-1,2){2}}
\put(4,11){\line(-1,2){2}} \put(12,3){\line(-1,2){2}}
\put(12,11){\line(-1,2){2}}


 \put(12,3){\line(-5,2){7.5}} \put(2,7){\line(5,-2){1.5}}
\put(10,7){\line(1,2){2}}



\put(12,3){\line(-1,1){8}}

\put(2,7){\line(1,2){2}}

\put(12,11){\line(-1,2){2}} 
\put(12,11){\line(-5,2){10}} 
\put(2,15){\line(1,-1){3.5}} \put(10,7){\line(-1,1){3.5}}

\put(2.2,3.1){\mbox{$a_{0}$}} \put(12.5,3.1){\mbox{$a_{1}$}}
\put(0.8,6){\mbox{$a_{2}$}} \put(10.5,6.5){\mbox{$a_{3}$}}
\put(2.3,11.2){\mbox{$a_{4}$}} \put(12.5,10){\mbox{$a_{5}$}}
\put(0.5,14.4){\mbox{$a_{6}$}} \put(11,14.4){\mbox{$a_{7}$}}

\put(5,1.3){\small \mbox{\boldmath{$B_0$}}}


\thicklines

\put(18,3){\line(1,0){8}} \put(18,3){\line(0,1){8}}
\put(18,11){\line(1,0){8}} \put(26,3){\line(0,1){8}}
\put(16,15){\line(1,0){8}} \put(16,7){\line(0,1){8}}
\put(16,7){\line(1,0){1.5}} \put(18.4,7){\line(1,0){3}}
\put(22.5,7){\line(1,0){1.5}} \put(24,7){\line(0,1){1.5}}
\put(24,9.5){\line(0,1){1.1}}

\put(24,11.5){\line(0,1){3.5}} \put(18,3){\line(-1,2){2}}
\put(18,11){\line(-1,2){2}} \put(26,3){\line(-1,2){2}}
\put(26,11){\line(-1,2){2}}

 \put(18,3){\line(3,2){6}}
 \put(18,11){\line(3,2){6}}

\thinlines

 \put(24,7){\line(1,2){2}}
 \put(18,3){\line(1,1){8}}
\put(16,7){\line(1,2){2}}


\put(16,7){\line(1,1){1.6}} \put(24,15){\line(-1,-1){3.7}}
\put(18.4,9.4){\line(1,1){1.2}}


\put(16.2,3.1){\mbox{$a_{0}$}} \put(26.5,3.1){\mbox{$a_{1}$}}
\put(14.8,6){\mbox{$a_{2}$}} \put(24.5,6.5){\mbox{$a_{3}$}}
\put(16.3,11.2){\mbox{$a_{4}$}} \put(26.5,10){\mbox{$a_{5}$}}
\put(14.5,14.4){\mbox{$a_{6}$}} \put(25,14.4){\mbox{$a_{7}$}}

\put(22,1.3){\small \mbox{\boldmath{$B_1$}}}


\thicklines

\put(32,3){\line(1,0){8}} \put(32,3){\line(0,1){8}}
\put(32,11){\line(1,0){8}} \put(40,3){\line(0,1){8}}
\put(30,15){\line(1,0){8}} \put(30,7){\line(0,1){8}}
\put(30,7){\line(1,0){1.5}} \put(32.4,7){\line(1,0){3}}
\put(36.5,7){\line(1,0){1.5}} \put(38,7){\line(0,1){3.5}}
\put(38,11.5){\line(0,1){3.5}} \put(32,3){\line(-1,2){2}}
\put(32,11){\line(-1,2){2}} \put(40,3){\line(-1,2){2}}
\put(40,11){\line(-1,2){2}}

\thinlines

\put(32,11){\line(3,2){6}}
\put(38,7){\line(-1,1){3.6}} \put(30,15){\line(1,-1){2.8}}
\put(33.5,11.1){\mbox{$\cdot$}}

\put(38,7){\line(1,2){2}}
\put(40,3){\line(-1,1){8}}
\put(30,7){\line(1,2){2}}

\put(32,3){\line(3,2){4.5}} \put(37,6.14){\mbox{$\cdot$}}
\put(37.35,6.35){\mbox{$\cdot$}}

\put(30.2,3.1){\mbox{$a_{0}$}} \put(40.5,3.1){\mbox{$a_{1}$}}
\put(28.8,6){\mbox{$a_{2}$}} \put(36.7,8.5){\mbox{$a_{3}$}}
\put(30.3,11.2){\mbox{$a_{4}$}} \put(40.5,10){\mbox{$a_{5}$}}
\put(28.5,14.4){\mbox{$a_{6}$}} \put(39,14.4){\mbox{$a_{7}$}}

\put(35,1.3){\small \mbox{\boldmath{$E$}}}

\put(11,0.5){\mbox{Figure 7}}

\end{picture}


\noindent \textbf{Type B\,:} To construct type B triangulations on the unit cube $[0,{1}/{3}]^3$, first take out two disjoint tetrahedra. The remaining polytope has eight triangular faces. We subdivide this polytope into 8 tetrahedra which have a common new vertex $b$ at the center $({1}/{6}, {1}/{6}, {1}/{6})$ (i.e., cone over the boundary with apex $b$). For instance, to form the triangulation $B_0$, we take out the tetrahedra $a_0a_2a_1a_4$ and $a_3a_5a_6a_7$ and subdivide the remaining into the following 8 tetrahedra (the boundary of the triangulated 3-ball $B_0$ is shown in  Figure 7)
\begin{equation}
ba_1a_2a_4, ba_1a_4a_5, ba_1a_3a_5, ba_3a_5a_6, ba_2a_3a_6, ba_2a_4a_6, ba_4a_5a_6, ba_1a_2a_3.
\end{equation}

For $B_1$ (respectively, $B_2$) we first take out the tetrahedra $a_0a_1a_3a_5$ and $a_2a_4a_6a_7$ (respectively, $a_1a_2a_3a_7$ and $a_0a_4a_5a_6$) and do the similar subdivision on the remaining polytope.

\medskip

\noindent \textbf{Type C\,:}
Choose a pair of diagonally opposite vertices, say $a_i$ and $a_j$, of a 2-face of the cube $[0, {1}/{3}]^3$. Let the neighbours of $a_i$ (resp., $a_j$) be $a_k$, $a_{\ell}$, $a_r$ (resp., $a_k$, $a_{\ell}$, $a_s$). First take out the tetrahedra
$a_ia_ka_{\ell}a_r$ and $a_ja_ka_{\ell}a_s$. Now, subdivide the remaining polytope as before adding the vertex $b$. Denote this triangulation by $C_{ij}$. The longest edge in $C_{ij}$ has length ${\sqrt{2}}/{3}$. The boundary of the triangulated 3-ball  $C_{05}$ is shown in Figure 7.

\medskip

\noindent \textbf{Type E\,}: Triangulation $E$ of $[0, 1/3]^{3}$ consists of the following 12 tetrahedra (the boundary is shown in Figure 7).
\begin{align}
ba_0a_1a_3, ba_0a_2a_3, ba_0a_1a_4, ba_0a_2a_4, ba_1a_3a_5, ba_1a_4a_5, \nonumber \\ ba_2a_3a_6, ba_2a_4a_6, ba_3a_5a_7, ba_3a_6a_7, ba_4a_5a_7, ba_4a_6a_7.
\end{align}

\medskip

We put together the cubes as shown in Figure 8 (a) to get a triangulation of $[0,1]^3$.

\setlength{\unitlength}{3mm}
\begin{picture}(48,28)(0,0)

\thicklines

\put(1,9){\line(0,1){16}} \put(16,3){\line(0,1){16}}
\put(4,10){\line(0,1){1}} \put(4,18){\line(0,1){1}}
\put(13,9){\line(0,1){1}} \put(13,17){\line(0,1){1}}

\put(4,3){\line(1,0){12}} \put(1,9){\line(1,0){12}}
\put(4,3){\line(-1,2){3}} \put(16,3){\line(-1,2){3}}

\thinlines

\put(3,5){\line(1,0){12}} \put(2,7){\line(1,0){12}}
\put(8,3){\line(-1,2){3}} \put(12,3){\line(-1,2){3}}

\put(4.2,3.5){\mbox{$A_{0}$}} \put(8.2,3.5){\mbox{$B_{2}$}}
\put(12.2,3.5){\mbox{$A_{1}$}} \put(3.2,5.5){\mbox{$B_{1}$}}
\put(7.2,5.5){\mbox{$A_{0}$}} \put(11.2,5.5){\mbox{$C_{27}$}}
\put(2.2,7.5){\mbox{$A_{1}$}} \put(6.2,7.5){\mbox{$C_{17}$}} \put(10.2,7.5){\mbox{$A_{0}$}}

\put(4,11){\line(1,0){12}} \put(1,17){\line(1,0){12}}
\put(4,11){\line(-1,2){3}} \put(16,11){\line(-1,2){3}}

\thinlines

\put(3,13){\line(1,0){12}} \put(2,15){\line(1,0){12}}
\put(8,11){\line(-1,2){3}} \put(12,11){\line(-1,2){3}}

\put(4.2,11.5){\mbox{$C_{56}$}} \put(8.2,11.5){\mbox{$A_{1}$}}
\put(12.2,11.5){\mbox{$B_{0}$}} \put(3.2,13.5){\mbox{$A_{1}$}}
\put(7.2,13.5){\mbox{$E$}} \put(11.2,13.5){\mbox{$A_{0}$}}
\put(2.2,15.5){\mbox{$B_{0}$}} \put(6.2,15.5){\mbox{$A_{0}$}} \put(10.2,15.5){\mbox{$C_{12}$}}

\put(4,19){\line(1,0){12}} \put(1,25){\line(1,0){12}}
\put(4,19){\line(-1,2){3}} \put(16,19){\line(-1,2){3}}

\thinlines

\put(3,21){\line(1,0){12}} \put(2,23){\line(1,0){12}}
\put(8,19){\line(-1,2){3}} \put(12,19){\line(-1,2){3}}

\put(4.2,19.5){\mbox{$A_{1}$}} \put(8.2,19.5){\mbox{$C_{06}$}}
\put(12.2,19.5){\mbox{$A_{0}$}} \put(3.2,21.5){\mbox{$C_{05}$}}
\put(7.2,21.5){\mbox{$A_{1}$}} \put(11.2,21.5){\mbox{$B_{1}$}}
\put(2.2,23.5){\mbox{$A_{0}$}} \put(6.2,23.5){\mbox{$B_{2}$}} \put(10.2,23.5){\mbox{$A_{1}$}}

\put(18,3.5){\mbox{(a)}} \put(22,3.5){\mbox{(b)}}


\thicklines

\put(21,13){\line(0,1){4.5}} \put(21,27){\line(0,-1){6.5}}
\put(20.8,18){\mbox{$\vdots$}} \put(20.8,19){\mbox{$\vdots$}}

\put(46,3){\line(0,1){6.5}} \put(46,17){\line(0,-1){4.5}}
\put(45.8,10){\mbox{$\vdots$}} \put(45.8,11){\mbox{$\vdots$}}

\put(26,16){\line(0,1){1}} \put(41,13){\line(0,1){1}}

\put(26,3){\line(1,0){9}} \put(37,3){\line(1,0){9}}
\put(21,13){\line(1,0){9}} \put(32,13){\line(1,0){9}}
\put(26,17){\line(1,0){9}} \put(37,17){\line(1,0){9}}
\put(21,27){\line(1,0){9}} \put(32,27){\line(1,0){9}}

\put(35.2,2.95){\mbox{$\dots$}}
\put(30.2,12.95){\mbox{$\dots$}}
\put(35.2,16.95){\mbox{$\dots$}}
\put(30.2,26.95){\mbox{$\dots$}}

\put(26,3){\line(-1,2){2.2}} \put(46,3){\line(-1,2){2.2}}
\put(26,17){\line(-1,2){2.2}} \put(46,17){\line(-1,2){2.2}}
\put(21,13){\line(1,-2){2.2}} \put(41,13){\line(1,-2){2.2}}
\put(21,27){\line(1,-2){2.2}} \put(41,27){\line(1,-2){2.2}}

\thinlines

\put(24,7){\line(1,0){9}} \put(35,7){\line(1,0){9}}
\put(23,9){\line(1,0){9}} \put(34,9){\line(1,0){9}}
\put(24,21){\line(1,0){9}} \put(35,21){\line(1,0){9}}
\put(23,23){\line(1,0){9}} \put(34,23){\line(1,0){9}}

\put(33.2,6.95){\mbox{$\dots$}}
\put(32.2,8.95){\mbox{$\dots$}}
\put(33.2,20.95){\mbox{$\dots$}}
\put(32.2,22.95){\mbox{$\dots$}}

\put(25,5){\line(1,0){9}} \put(36,5){\line(1,0){9}}
\put(22,11){\line(1,0){9}} \put(33,11){\line(1,0){9}}
\put(25,19){\line(1,0){9}} \put(36,19){\line(1,0){9}}
\put(22,25){\line(1,0){9}} \put(33,25){\line(1,0){9}}

\put(34.2,4.95){\mbox{$\dots$}}
\put(31.2,10.95){\mbox{$\dots$}}
\put(34.2,18.95){\mbox{$\dots$}}
\put(31.2,24.95){\mbox{$\dots$}}

\put(34,3){\line(-1,2){2.2}} \put(38,3){\line(-1,2){2.2}}
\put(34,17){\line(-1,2){2.2}} \put(38,17){\line(-1,2){2.2}}
\put(29,13){\line(1,-2){2.2}} \put(33,13){\line(1,-2){2.2}}
\put(29,27){\line(1,-2){2.2}} \put(33,27){\line(1,-2){2.2}}

\put(30,3){\line(-1,2){2.2}} \put(42,3){\line(-1,2){2.2}}
\put(30,17){\line(-1,2){2.2}} \put(42,17){\line(-1,2){2.2}}
\put(25,13){\line(1,-2){2.2}} \put(37,13){\line(1,-2){2.2}}
\put(25,27){\line(1,-2){2.2}} \put(37,27){\line(1,-2){2.2}}



\put(23.2,8.2){\mbox{$.$}} \put(23.4,7.8){\mbox{$.$}}
\put(27.2,8.2){\mbox{$.$}} \put(27.4,7.8){\mbox{$.$}}
\put(31.2,8.2){\mbox{$.$}} \put(31.4,7.8){\mbox{$.$}}
\put(35.2,8.2){\mbox{$.$}} \put(35.4,7.8){\mbox{$.$}}
\put(39.2,8.2){\mbox{$.$}} \put(39.4,7.8){\mbox{$.$}}
\put(43.2,8.2){\mbox{$.$}} \put(43.4,7.8){\mbox{$.$}}

\put(23.2,22.2){\mbox{$.$}} \put(23.4,21.8){\mbox{$.$}}
\put(27.2,22.2){\mbox{$.$}} \put(27.4,21.8){\mbox{$.$}}
\put(31.2,22.2){\mbox{$.$}} \put(31.4,21.8){\mbox{$.$}}
\put(35.2,22.2){\mbox{$.$}} \put(35.4,21.8){\mbox{$.$}}
\put(39.2,22.2){\mbox{$.$}} \put(39.4,21.8){\mbox{$.$}}
\put(43.2,22.2){\mbox{$.$}} \put(43.4,21.8){\mbox{$.$}}

\put(26.5,3.5){\mbox{$A_{0}$}} \put(30.5,3.5){\mbox{$A_{1}$}}
\put(38.5,3.5){\mbox{$A_{0}$}} \put(42.5,3.5){\mbox{$A_{1}$}}

\put(25.5,5.5){\mbox{$A_{1}$}} \put(29.5,5.5){\mbox{$A_{0}$}}
\put(37.5,5.5){\mbox{$A_{1}$}} \put(41.5,5.5){\mbox{$A_{0}$}}

\put(23.5,9.5){\mbox{$A_{0}$}} \put(27.5,9.5){\mbox{$A_{1}$}}
\put(35.5,9.5){\mbox{$A_{0}$}} \put(39.5,9.5){\mbox{$A_{1}$}}

\put(22.5,11.5){\mbox{$A_{1}$}} \put(26.5,11.5){\mbox{$A_{0}$}}
\put(34.5,11.5){\mbox{$A_{1}$}} \put(38.5,11.5){\mbox{$A_{0}$}}

\put(26.5,17.5){\mbox{$A_{1}$}} \put(30.5,17.5){\mbox{$A_{0}$}}
\put(38.5,17.5){\mbox{$A_{1}$}} \put(42.5,17.5){\mbox{$A_{0}$}}

\put(25.5,19.5){\mbox{$A_{0}$}} \put(29.5,19.5){\mbox{$A_{1}$}}
\put(37.5,19.5){\mbox{$A_{0}$}} \put(41.5,19.5){\mbox{$A_{1}$}}

\put(23.5,23.5){\mbox{$A_{1}$}} \put(27.5,23.5){\mbox{$A_{0}$}}
\put(35.5,23.5){\mbox{$A_{1}$}} \put(39.5,23.5){\mbox{$A_{0}$}}

\put(22.5,25.5){\mbox{$A_{0}$}} \put(26.5,25.5){\mbox{$A_{1}$}}
\put(34.5,25.5){\mbox{$A_{0}$}} \put(38.5,25.5){\mbox{$A_{1}$}}

\put(1,0){\mbox{Figure\,8\,: {\bf Triangulations of the cube.} (a) {\bf Mesh} \boldmath{$\leq \frac{\sqrt{2}}{3}$},
(b) {\bf mesh} \boldmath{$\leq \frac{1}{m\sqrt{2}}$}}}
\end{picture}

\begin{lemma} \label{lemma:77-vertex}
 There exists a $77$-vertex triangulation $X$ of $[0,1]^3$ which satisfies the following.
 \begin{enumerate}[{\rm (a)}]
  \item Each $2$-face of $[0,1]^3 $ contains $16 $ vertices.

  \item There are $64$ vertices of the form $({i}/{3}, {j}/{3}, {k}/{3})$, $0 \leq i,j,k \leq 3$.

  \item There are $13$ vertices are from the set $\left\{\left(\frac{i + 1/2}{3}, \frac{j + 1/2}{3}, \frac{k + 1/2}{3}\right) \, : \, 0 \leq i,j,k \leq 2\right\}$ of $27$ points.

  \item The triangulation of $\{1\} \times [0,1] \times [0,1] $ is the translation of the triangulation
     of $\{0\}\times [0,1] \times [0,1]$.
        Similar thing holds true for the other pairs of opposite faces of the cube.
   \item The diameter of each tetrahedron is at most ${\sqrt{2}}/{3}$.
 \end{enumerate}
\end{lemma}

\begin{proof}
 We first subdivide the cube $[0,1]^3 $ into $27$ small cubes (of equal sizes) of type $[0, {1}/{3}]^3 $.  Then we subdivide each cube consistently  so that on the common face of two cubes the subdivisions coming from  each cube agree. We take the subdivision of $[0,1]^2 \times [0, {1}/{3}]$ as the union of $A_0$, $B_2 + ({1}/{3}, 0,0)$, $A_1 + ({2}/{3}, 0,0)$, $B_1+(0, {1}/{3},0)$, $A_0+ ({1}/{3},{1}/{3},0)$, $C_{27} + ({2}/{3}, {1}/{3}, 0)$, $A_1 +(0,0, {2}/{3})$, $C_{17} +({1}/{3}, {2}/{3}, 0)$, $A_0+({2}/{3}, {2}/{3}, 0)$. Similarly, we subdivide $[0,1]^2 \times [{1}/{3}, {2}/{3}]$ and $[0,1]^2 \times [{2}/{3}, 1]$. This is shown in Figure 8 (a). Observe that the subdivision of $[0,1]^2 \times \{ 1\}$ is the translation of  subdivision of $[0,1]^1 \times \{0\} $. Similarly, other pairs of the opposite faces of the cube $[0,1]^3$  have subdivision that are translations of each other.

 There are $64$ vertices on the boundary of the cube $[0,1]^3$ and $13$ vertices lie in the interior of it.
 Thus, the total number of vertices is $77$.  We used fourteen $A$-type cubes, six $B$-type, six $C$-type and one
 of $E$-type. Therefore, the number of tetrahedra is $14 \times 5 + 6 \times 10 + 6 \times 10 + 1 \times 12 =202$. Let the number of 2-simplices be $f_2$. Counting the number of ordered pair $(\sigma, \gamma)$, where $\gamma$ is a 2-simplex in a 3-simplex $\sigma$, we get $202\times 4 = 108\times 1 + (f_2 - 108)\times 2$. Thus, $f_2 = 458$. Since the Euler characteristic of a triangulated 3-ball is 1, it follows that the number of  edges is $332$.

From construction, the lengths of edges are $1/3$, ${1}/{3}$, ${\sqrt{2}}/{3}$ and ${1}/{(2\sqrt{3})}$ respectively. This completes the proof.
\end{proof}

\begin{corollary}
There exists a contact triangulation of $(T^3, \xi_1)$ with $40$ vertices and $202$ tetrahedra.
\end{corollary}

\begin{proof}
Observe that the triangulation $X$ (in Lemma \ref{lemma:77-vertex}) on the cube induces a triangulation of the 3-torus as the triangulation of opposite faces match. It has 202 tetrahedra. Clearly, the number of vertices is $27+13=40$.
We also note that diameter of each tetrahedron strictly less than ${1}/{2}$.
This ensures that the contact structure $\xi_1$ restricted to each tetrahedron
is tight as the smallest overtwisted disk has diameter greater than ${1}/{2}$.
This also implies that each 2-face is not an overtwisted disk. Now, we perturb
all the triangles to make the edges Legendrian.
\end{proof}


\begin{proof}[Proof Of Theorem \ref{3-torus_thm}]
We recall that the smallest overtwisted disk has the radius ${r_0}/{n}$. Thus, subdivide the unit cube $[0,1]^3 $ into $8n^3$ small cubes of the size $\frac{1}{2n} \times \frac{1}{2n} \times \frac{1}{2n}$ as shown in  Figure 8 (b). We then subdivide each cube into tetrahedra by using type A triangulations. In particular, we use $\frac{3}{2n}A_0$ and $\frac{3}{2n} A_1$ where the factor ${3}/{2n}$ in front of $A_i$ means scaling of $A_i$ in all coordinate directions by factor of ${3}/{2n}$. We first take triangulation $\frac{3}{2n}A_0$ for the cube $[0, {1}/{2n}]^3$. Now if two small cubes in $[0,1]^3$ have a common face then we take $\frac{3}{2n}A_0$ for one of the cubes and $\frac{3}{2n}A_1$ for the other. Thus, there is a unique consistent arrangement of smaller cubes. Moreover, we get a subdivision of the cube $[0,1]^3$ so that the induced triangulations on opposite faces match. Therefore, we get a triangulation of the 3-torus. Observe that the triangulation of the cube has $(2n+1)^3$ vertices and $5\times (2n)^3$ tetrahedra. This implies that the triangulation of the torus has $(2n)^3$ vertices and $5\times (2n)^3$ \, 3-simplices.

The diameter of each tetrahedron is less than ${1}/({\sqrt{2}n})$. Thus, no tetrahedron contains an overtwisted disk as ${1}/({\sqrt{2}n}) < {2r_0}/{n}$. Now, we perturb all the triangles slightly to make the edges Legendrian.
\end{proof}

{\small

}

\end{document}